\documentclass{article}

\usepackage[english]{babel}

\usepackage[letterpaper,top=2cm,bottom=2cm,left=3cm,right=3cm,marginparwidth=1.75cm]{geometry}

\usepackage{amsmath}
\usepackage{graphicx}
\usepackage{authblk}
\usepackage[colorlinks=true, allcolors=blue]{hyperref}
\usepackage{comment}

\usepackage{amsmath, amsthm, amssymb, graphicx}
\usepackage{float}

\newtheorem{theorem}{Theorem}[section]
\newtheorem{lemma}[theorem]{Lemma}

\newtheorem{proposition}[theorem]{Proposition}
\theoremstyle{definition}
\newtheorem{definition}[theorem]{Definition}
\newtheorem{remark}[theorem]{Remark}
\newtheorem{corollary}[theorem]{Corollary}

\usepackage{tikz}
\usetikzlibrary{arrows.meta}
\usetikzlibrary{positioning,calc,shapes.geometric}

\title{$\partial$-invariant path generators for digraphs}
\author[1]{Zhenzhi Li}
\author[2]{Wujie Shen}
\affil[1]{School of Mathematical Sciences, East China Normal University,  \texttt{zzli@stu.ecnu.edu.cn}}
\affil[2]{Yau Mathematical Sciences Center and Department of Mathematics, Tsinghua University, \texttt{shenwj22@mails.tsinghua.edu.cn}}

\begin{document}
\maketitle

\begin{abstract}
We study the structure of the space $\Omega_3(G)$ of $\partial$-invariant 3-paths in a directed graph $G$. We prove that $\Omega_3(G)$ admits a basis consisting of trapezohedral paths $\tau_m$ ($m \ge 2$) and their merging images. Moreover, we provide an explicit construction of such a basis and, as a consequence, obtain an algorithm with time complexity $O(|V(G)|^5)$ for computing the dimension and a basis of $\Omega_3(G)$ for any finite digraph. %The proof is based on a structural analysis of minimal $\partial$-invariant clusters and a combinatorial construction using colored graphs.

\end{abstract}

\section{Introduction}

%GLMY理论介绍，这里应该多加几个文献
The concept of \emph{path homology} for digraphs, now commonly called GLMY homology, was introduced by Grigor’yan, Lin, Muranov, and Yau in \cite{GrigorYan2013homologies}. It provides an algebraic and topological framework for graph theory. The theory is closely analogous to classical homology theory in topology, incorporating key concepts such as chain complexes, homology, and Künneth formulas. For a more extensive bibliography, we refer the reader to \cite{grigoryan_hodge,GrigorYan2019homology,GrigorYan2014homotopy,GrigorYan2020path,GrigorYan2017kunneth}.

%GLMY理论在应用数学里的作用
The GLMY homology finds application across diverse fields within applied mathematics, including artificial intelligence (AI), material science, chemistry and biology (e.g. \cite{chen2023path,liu2023neighborhood,liu2023persistent}). A central problem and a primary computational challenge is to design fast algorithms for computing homology generators for large, finite digraphs. A mathematical study of the structure and generators of GLMY homology would be instrumental in addressing this problem.

%理论具体介绍以及研究进展。这里直接抄的别的文献还没改

The GLMY theory has advanced significantly in recent years. Recent developments include the application of discrete Morse theory to digraphs, a theory of covering digraphs linking path homology to group homology via Cayley digraphs, and structural analyses of path and cellular complexes under strongly regular conditions \cite{di2024path,lin2021discrete,tang2025minimal,tang2024cellular}. Ivanov–Pavutnitskiy \cite{ivanov2024simplicial} recast path homology into a simplicial framework, while Fu–Ivanov \cite{fu2024path} constructed explicit bases for path chain spaces without multisquares. Li–Muranov–Wu–Yau \cite{li2025singular} developed cubical and simplicial homology theories for digraphs and quivers. The theory also interfaces with analysis on graphs through path Laplacians and Hodge-type decompositions \cite{grigoryan_hodge,grigoryan2025eigenvalues}, making it relevant to the study of directed network dynamics. 

For a digraph $G$, it is known that $\dim \Omega_0(G)$ coincides with the number of vertices, $\dim \Omega_1(G)$ with the number of arrows, and $\dim \Omega_2(G)$ with the total number of triangles, squares, and twice the number of double arrows. Recall that $G$ is said to have \emph{no double-arrows} if $a\rightarrow b$ and $b\rightarrow a$ cannot hold simultaneously for any two vertices $a\neq b$, and \emph{no multisquares} if between any two vertices at directed distance two, there exist at most two distinct shortest directed paths. Under these conditions, Grigoryan established in \cite[Theorem 2.10]{grigoryan2022advances} an explicit basis for $\Omega_3(G)$ consisting of trapezohedrons and their images under digraph morphisms. In contrast, for digraphs that containing multisquares, the dimension of $\Omega_n(G)$ remains undetermined in general.

%引入我们的结论

In the present paper, we prove this without imposing the no double-arrow and no multisquare conditions, thereby extending the result. Moreover, we have an explicit algorithm for computing the dimension and constructing a basis of $\Omega_3(G)$, which answers Problem 2.11 and 2.12 in \cite{grigoryan2022advances} (see Section \ref{sec:prelim} for the relevant definitions):

\begin{theorem}\label{thm:main}
Let $G$ be a digraph. There exists a basis of $\Omega_3(G)$ consisting of trapezohedral paths $\tau_m$ with $m \ge 2$ and their merging images. Moreover, there is an algorithm with time complexity $O(|V|^5)$ for determining the dimension and a basis of $\Omega_3(G)$, where $|V|$ is the number of vertices of $G$.
\end{theorem}

Note that in this paper, we provide an \emph{explicit} construction of a basis for $\Omega_3$, in contrast to the recursive algorithms found in previous research (e.g. \cite{tang2025minimal}). This result is stated in Theorem \ref{thm:main2} in the beginning of Section \ref{sec:dimandbasis}.

The paper is structured as follows. In Section \ref{sec:prelim}, we recall the basic theory of GLMY homology and introduce the notations used throughout the paper, including trapezohedrons and merging images. Section \ref{sec:structureomega3} is devoted to the proof of the first statement in Theorem \ref{thm:main}. We begin in Section \ref{subsec1:structureomega3} with essential examples of merging images of trapezohedrons. Following the approach in \cite[Theorem 2.10]{grigoryan2022advances} together with some additional arguments, we complete the proof in Section \ref{subsec2:structureomega3}. In Section \ref{sec:dimandbasis}, we study the dimension and basis of $\Omega_3$. More specifically, Sections \ref{sec:basis1} through \ref{sec:basis4} provide a detailed analysis, and in Section \ref{sec:algorithm} we present an algorithm for constructing a basis and analyze its time complexity.

\subsection*{Acknowledgments}

We express our deep gratitude to Professors Yong Lin and Shing-Tung Yau for introducing the problem and unwavering support. We also thank Jiarui Feng, Jianhui Li and Haohang Zhang for their helpful comments and discussions, which improved the presentation.

\section{Preliminaries}\label{sec:prelim}

In this section, we give a brief introduction of GLMY theory and present some concepts and preliminary mathematical results that will be used in the subsequent sections. We refer the readers to \cite{grigoryan2022advances}.

\subsection{Notations about digraphs}

\begin{definition}[Digraphs]
A \textit{digraph} (directed graph) is a pair $G = (V, E)$ of a set $V=V(G)$ of vertices and $E=E(G) \subset \{V \times V \setminus \mathrm{diag}\}$ is a set of arrows (directed edges). If $(i, j) \in E$ then we write $i \to j$.
\end{definition}

\begin{definition}[Neighborhood]
For a digraph $G$ and a vertex $v \in V(G)$, define the \emph{out-neighborhood} $N^+(v)$ to be the set of vertices $w \in V(G)$ such that $v \rightarrow w$. Similarly, define the \emph{in-neighborhood} $N^-(v)$ to be the set of vertices $w \in V(G)$ such that $w \rightarrow v$.
\end{definition}

\begin{definition}[Induced subgraphs]
For a digraph $G$ and a subset $A \subseteq V(G)$, define the \emph{induced subgraph} $G[A]$ of $A$ to be the subgraph of $G$ with vertex set $A$ and edge set consisting of all edges in $E(G)$ that start and end in $A$. 
\end{definition}

Next, for disjoint subsets $A, B \subseteq V(G)$, we define the induced digraph from $A$ to $B$ as follows, this definition will be used in Section \ref{sec:dimandbasis}.

\begin{definition}[Induced digraph from $A$ to $B$]\label{def:ind}
For disjoint subsets $A, B \subseteq V(G)$, we define the induced digraph $\operatorname{Ind}(A,B)$ from $A$ to $B$ as follows: its vertex set is $A\cup B$ and its edge set consists of all edges of $G$ that start in $A$ and end in $B$.
\end{definition}

\subsection{Paths and the boundary operator $\partial$}

\begin{definition}[Elementary $p$-paths]
For any $p \geq 0$, an elementary $p$-path is any sequence $i_0, i_1, \dots, i_p$ of $p + 1$ vertices of $V$ (allowing repetitions). Fix a field $\mathbb{K}$ and denote by $\Lambda_p = \Lambda_p(V, \mathbb{K})$ the $\mathbb{K}-$linear space that consists of all formal $\mathbb{K}-$linear combinations of elementary $p$-paths in $V$. Any element of $\Lambda_p$ is called a $p$-path.
\end{definition}

An elementary $p$-path $i_0, \dots, i_p$ as an element of $\Lambda_p$ will be denoted by $e_{i_0 \dots i_p}$. For example, we have
$$\Lambda_0 = \langle e_i : i \in V \rangle, \quad \Lambda_1 = \langle e_{ij} : i, j \in V \rangle, \quad \Lambda_2 = \langle e_{ijk} : i, j, k \in V \rangle.$$

Here the angle brackets $\langle \cdot \rangle$ denote the linear span of the enclosed elements. Any $p$-path $u\in\Lambda_p$ can be written in a form 
$$u = \sum_{i_0, i_1, \dots, i_p \in V} u^{i_0 i_1 \dots i_p} e_{i_0 i_1 \dots i_p},$$
where $u^{i_0 i_1 \dots i_p} \in \mathbb{K}$.

\begin{definition}[Boundary operator]
For any $p \geq 1$ define a linear boundary operator $\partial : \Lambda_p \to \Lambda_{p-1}$ by
$$\partial e_{i_0 \dots i_p} = \sum_{q=0}^p (-1)^q e_{i_0 \dots \hat{i_q} \dots i_p},$$
where $\hat{i_q}$ indicates that the index $i_q$ is omitted. Set $\Lambda_{-1} = {0}$ and define $\partial : \Lambda_0 \to \Lambda_{-1}$ by $\partial = 0$.
\end{definition} 

\begin{definition}[Regular $p$-paths]
An elementary $p$-path $i_{0}\dots i_{p}$ is called \textit{regular} if $i_{k} \neq i_{k+1}$ for all $k = 0,\dots,p-1$. Otherwise, it is called \emph{irregular}. Denote by $\mathcal{I}_p$ the subspace of $\Lambda_p$ generated by all irregular $p$-paths $e_{i_0\dots i_p}$.
\end{definition}

If $e_{i_0 \dots i_p}$ is irregular, then $i_k = i_{k+1}$ for some $k$. In this case, $\partial \mathcal{I}_p \subset \mathcal{I}_{p-1}$. Hence, $\partial$ is well-defined on the quotient spaces $\mathcal{R}_p := \Lambda_p / \mathcal{I}_p$.
\begin{comment}
We have
$$\begin{aligned}
\partial e_{i_0 \dots i_p} &= \sum_{q=0}^p (-1)^q e_{i_0 \dots \hat{i_q} \dots i_p} \\
&=\sum_{q\ne k,k+1} (-1)^q e_{i_0 \dots \hat{i_q} \dots i_p} + (-1)^k e_{i_0 \dots i_{k-1} i_{k+1} i_{k+2} \dots i_p} + (-1)^{k+1} e_{i_0 \dots i_{k-1} i_k i_{k+2} \dots i_p} \\
&=\sum_{q\ne k,k+1} (-1)^q e_{i_0 \dots \hat{i_q} \dots i_p} \in \mathcal{I}_{p-1}.
\end{aligned}$$

Therefore, $\partial \mathcal{I}_p \subset \mathcal{I}_{p-1}$. 
\end{comment}
%Hence, $\partial$ is well-defined on the quotient spaces $\mathcal{R}_p := \Lambda_p / \mathcal{I}_p$.

\subsection{$\mathcal{A}_p$ and $\Omega_p$}

\begin{definition}[Allowed elementary $p$-paths]
Let $G = (V, E)$ be a digraph. An elementary $p$-path $i_0 \dots i_p$ on $V$ is called \textit{allowed} if $i_k \to i_{k+1}$ for any $k = 0, \dots, p-1$, and \textit{non-allowed} otherwise.
\end{definition}

Let $\mathcal{A}_p = \mathcal{A}_p(G)$ be $\mathbb{K}$-linear subspace of $\Lambda_p$ spanned by allowed elementary $p$-paths:
$$\mathcal{A}_p = \langle e_{i_0 \dots i_p} : i_0 \dots i_p \text{ is allowed} \rangle.$$
The elements of $\mathcal{A}_p$ are called \textit{allowed $p$-paths}. Since any allowed path is regular, we have $\mathcal{A}_p \subset \mathcal{R}_p$.

\begin{definition}[$\partial$-invariant $p$-paths]
Consider the following subspace $\Omega_p$ of $\mathcal{A}_p$:
$$\Omega_p \equiv \Omega_p(G) := \{ u \in \mathcal{A}_p : \partial u \in \mathcal{A}_{p-1} \}.$$
The elements of $\Omega_p$ are called $\partial$-invariant $p$-paths.
\end{definition}

\begin{remark}
    Recall that $\mathcal{R}_p := \Lambda_p / \mathcal{I}_p$. Consequently, an elementary path $e_{i_0\ldots i_p}$ vanishes in $\mathcal{R}_p$ whenever there exist two consecutive vertices $i_k$ and $i_{k+1}$ with $i_k = i_{k+1}$. For example, $\partial e_{abab} = e_{bab} - e_{aab} + e_{abb} - e_{aba} = e_{bab} - e_{aba}$, so $e_{abab} \in \Omega_3$.
\end{remark}

\subsection{Cluster basis in $\Omega_p$}

\begin{definition}[Cluster]
A $p$-path $\omega = \sum \omega^{i_0 \dots i_p} e_{i_0 \dots i_p}$ is called an $(a,b)$-\textit{cluster} if all the elementary paths $e_{i_0 \dots i_p}$ with non-zero values of $\omega^{i_0 \dots i_p}$ have $i_0 = a$ and $i_p = b$. A path $\omega$ is called a \textit{cluster} if it is an $(a,b)$-cluster for some $a,b\in V(G)$.
\end{definition}

\begin{lemma}\label{lem:cluster}
Any $\partial$-invariant $p$-path is a sum of $\partial$-invariant clusters.
\end{lemma}

\begin{proof}
This is \cite[Lemma 2.2]{grigoryan2022advances}. For the sake of completeness, we present the proof here. Let $\omega \in \Omega_p$. For any vertices $a, b \in V$, denote by $\omega_{a,b}$ the sum of all terms $\omega^{i_0 \dots i_p} e_{i_0 \dots i_p}$ with $i_0 = a$ and $i_p = b$.
Then $\omega_{a,b}$ is a cluster and $\omega = \sum_{a,b \in V} \omega_{a,b}$, that is, $\omega$ is a sum of clusters. 

Since $\omega$ is allowed, also all non-zero terms $\omega^{i_0 \dots i_p} e_{i_0 \dots i_p}$ are allowed, whence $\omega_{a,b}$ is also allowed. Let us prove that $\partial \omega_{a,b}$ is allowed, which will yield the $\partial$-invariance of $\omega_{a,b}$. The path $\omega_{a,b}$ is a linear combination of allowed paths of the form $e_{a i_1 \dots i_{p-1} b}$. We have
$$\partial e_{a i_1 \dots i_{p-1} b} = e_{i_1 \dots i_{p-1} b} + (-1)^p e_{a i_1 \dots i_{p-1}} + \sum_{k=1}^{p-1} (-1)^k e_{a i_1 \dots \hat{i_k} \dots i_{p-1} b}.$$

The terms $e_{i_1 \dots i_{p-1} b}$ and $e_{a i_1 \dots i_{p-1}}$ are clearly allowed, while among the terms $e_{a i_1 \dots \hat{i_k} \dots i_{p-1} b}$ there may be non-allowed. In the full expansion of
$$\partial \omega = \sum_{a,b \in V} \partial \omega_{a,b}$$
all non-allowed terms must be cancelled out by some terms in $\omega_{a,b}$. Indeed, since all the terms $e_{a i_1 \dots \hat{i_k} \dots i_{p-1} b}$ form a $(a,b)$-cluster, they cannot cancel with terms containing different values of $a$ or $b$. Therefore, they have to cancel already within $\partial \omega_{a,b}$, which implies that $\partial \omega_{a,b}$ is allowed.
\end{proof}

\begin{definition}
For any $p$-path $\omega = \sum \omega^{i_0 \dots i_p} e_{i_0 \dots i_p}$, define $e(\omega)$ to be the set of $p$-paths with nonzero coefficient $\omega^{i_0 \dots i_p}$, and define its \textit{width} $\|\omega\|$ to be the cardinality of $e(\omega)$.
\end{definition}

\begin{definition}
A $\partial$-invariant path $\omega$ is called \textit{minimal} if there is no nonempty proper subset of $e(\omega)$ whose nonzero linear combination is also $\partial$-invariant.
\end{definition}

\begin{lemma}\cite[Lemma 2.4]{grigoryan2022advances}\label{lem:minimalcluster}
Every $\partial$-invariant cluster is a sum of minimal $\partial$-invariant clusters.
\end{lemma}

\begin{proof}
Let $\omega$ be a $\partial$-invariant cluster that is not minimal. Then we have $\omega = \sum_{k=1}^n \omega^{(k)},$ where each $\omega^{(k)}$ is a $\partial$-invariant path with $\|\omega^{(k)}\| < \|\omega\|$. By Lemma \ref{lem:cluster}, each $\omega^{(k)}$ is a sum of clusters $\omega^{(k)}_{a,b}$, and it is clear from the definition of $\omega^{(k)}_{a,b}$ that$\|\omega^{(k)}_{a,b}\| \leq \|\omega^{(k)}\|.$ Hence, we can replace each $\omega^{(k)}$ by $\sum_{a,b} \omega^{(k)}_{a,b}$ and, hence, assume without loss of generality that all terms $\omega^{(k)}$ are $\partial$-invariant clusters. If some $\omega^{(k)}$ in this sum is not minimal then we replace it further with a sum of $\partial$-invariant clusters with smaller widths. Continuing this procedure we obtain in the end a representation $\omega$ as a sum of minimal $\partial$-invariant clusters.
\end{proof}

\begin{proposition}\cite[Proposition 2.5]{grigoryan2022advances}\label{pro:cluster}
The space $\Omega_p$ has a basis that consists of minimal $\partial$-invariant clusters.
\end{proposition}

\begin{proof}
Let $\mathcal{M}$ denote the set of all minimal $\partial$-invariant clusters in $\Omega_p$. By Lemma \ref{lem:minimalcluster}, every element of $\Omega_p$ is a sum of elements of $\mathcal{M}$. Selecting a maximal linearly independent subset of $\mathcal{M}$, we obtain a basis in $\Omega_p$.
\end{proof}

\subsection{Merging map}

\begin{definition}[Digraph morphisms]
A mapping $f : X \to Y$ between the sets of vertices of $X$ and $Y$ called a \textit{digraph map (or morphism)} if for every directed edge $a \to b$ on $X$, we have either $f(a) \to f(b)$ or $f(a) = f(b)$ on $Y$. In other words, any arrow of $X$ under the mapping $f$ either goes to an arrow of $Y$ or collapses to a vertex of $Y$.
\end{definition}

Let $X$ be a directed graph whose vertex set is partitioned into disjoint subsets $A_1, A_2, \dots, A_n$. Define a directed graph $Y$ with vertices $a_1, a_2, \dots, a_n$. Define a map $f : X \to Y$ by $f(x) = a_i \text{ for all } x \in A_i.$ If
$$a_i \to a_j \text{ in } Y \quad \text{if and only if there exist } x \in A_i \text{ and } y \in A_j \text{ such that } x \to y \text{ in } X,$$ 
then the map $f$ is called a \textit{merging map}.

\medskip
Any digraph morphism $f : X \to Y$ induces a mapping $f_* : \Lambda_n(X) \to \Lambda_n(Y)$ as follows: first set $f_*(e_{i_0 \dots i_n}) = e_{f(i_0) \dots f(i_n)},$ and then extend $f_*$ by linearity to all of $\Lambda_n(X)$.

\begin{proposition}\label{pro:morphism}
Let $f : X \to Y$ be a digraph morphism. Then the induced mapping $f_* : \Lambda_n(X) \to \Lambda_n(Y)$ extends to a chain mapping $f_* : \Omega_n(X) \to \Omega_n(Y)$. 
\end{proposition}

\begin{proof}
If $e_{i_0 \dots i_n}$ is irregular then $f_*(e_{i_0 \dots i_n})$ is also irregular. Therefore, $f_*$ maps the space $\mathcal{I}_n(X)$ of irregular paths on $X$ into $\mathcal{I}_n(Y)$. It follows that $f_*$ maps $\mathcal{R}_n(X) = \Lambda_n(X)/\mathcal{I}_n(X)$ into $\mathcal{R}_n(Y)$. If $e_{i_0 \dots i_n}$ is allowed then $i_k \to i_{k+1}$ for all $k$, which implies that either $f(i_k) \to f(i_{k+1})$ for all $k$ and, hence, $f_*(e_{i_0 \dots i_n})$ is also allowed, or $f(i_k) = f(i_{k+1})$ for some $k$ so that $f_*(e_{i_0 \dots i_n})$ is irregular, thus $f_*(e_{i_0 \dots i_n}) = 0$. Hence, $f_*$ maps the space $\mathcal{A}_n(X)$ of allowed paths into $\mathcal{A}_n(Y)$. Clearly, $f_*$ commutes with $\partial$, which implies that $f_*$ maps $\Omega_n(X)$ into $\Omega_n(Y)$ and $f_*$ is a chain mapping. 
\end{proof}

\subsection{Trapezohedron}

\begin{definition}
For any integer $m \geq 2$, define a \textit{trapezohedron} $T_m$ (cf. Fig.\ref{fig:T_m}) of order $m$ as follows: 

$T_m$ is a digraph of $2m + 2$ vertices
$$a, b, i_0, \dots, i_{m-1}, j_0, \dots, j_{m-1}$$
and $4m$ arrows
$$a \to i_k \to j_k \to b, \quad i_{k+1} \to j_k$$
for all $k = 0, \dots, m - 1 \mod m$.
\end{definition}

\begin{figure}[htbp]
    \centering
    \begin{tikzpicture}[>= {Stealth[scale=1.5]}, scale=0.9, font=\small]
    % ========== 参数定义 ==========
    \def\xR{4}                  % 长轴半径
    \def\yR{0.4}               % 短轴半径
    \def\jY{-3.2}             % 下椭圆中心y坐标（随扁度微调）
    
    % ========== 绘制虚线椭圆 ==========
    \draw[dashed] (0,0) ellipse [x radius=\xR, y radius=\yR];      % i椭圆
    \draw[dashed] (0,\jY) ellipse [x radius=\xR, y radius=\yR];   % j椭圆
    
    % ========== 节点坐标（椭圆下半部） ==========
    % i 行：左侧180°,210°,240°；右侧300°,330°,360°
    \coordinate (im1)  at ({\xR*cos(180)}, {\yR*sin(180)});
    \coordinate (i0)   at ({\xR*cos(220)}, {\yR*sin(220)});
    \coordinate (i1)   at ({\xR*cos(240)}, {\yR*sin(240)});
    \coordinate (ikm1) at ({\xR*cos(300)}, {\yR*sin(300)});
    \coordinate (ik)   at ({\xR*cos(320)}, {\yR*sin(320)});
    \coordinate (ikp1) at ({\xR*cos(360)}, {\yR*sin(360)});
    
    % j 行：对应角度，中心下移 \jY
    \coordinate (jm1)  at ({\xR*cos(180)}, {\jY+\yR*sin(180)});
    \coordinate (j0)   at ({\xR*cos(220)}, {\jY+\yR*sin(220)});
    \coordinate (j1)   at ({\xR*cos(240)}, {\jY+\yR*sin(240)});
    \coordinate (jkm1) at ({\xR*cos(300)}, {\jY+\yR*sin(300)});
    \coordinate (jk)   at ({\xR*cos(320)}, {\jY+\yR*sin(320)});
    \coordinate (jkp1) at ({\xR*cos(360)}, {\jY+\yR*sin(360)});
    
    % a 与 b
    \coordinate (A) at (0,2);
    \coordinate (B) at (0,\jY-2);   % 随下椭圆位置调整
    
    % ========== 绘制小黑点 ==========
    \foreach \p in {im1,i0,i1,ikm1,ik,ikp1, jm1,j0,j1,jkm1,jk,jkp1, A, B}
        \draw[fill=black] (\p) circle (2pt);
    
    % ========== 省略号 ==========
    \node at ({\xR*cos(270)}, {\yR*sin(270)-0.5}) {$\boldsymbol{\cdots}$};
    \node at ({\xR*cos(270)}, {\jY+\yR*sin(270)+0.4}) {$\boldsymbol{\cdots}$};
    
    % ========== 标签 ==========
    % i 点
    \node[label={[label distance=0.5mm]below left:$i_{m-1}$}] at (im1) {};
    \node[label={[label distance=0.5mm]below right:$i_0$}] at (i0)  {};
    \node[label={[label distance=0.5mm]below right:$i_1$}] at (i1)  {};
    \node[label={[label distance=0.5mm]below left:$i_{k-1}$}] at (ikm1) {};
    \node[label={[label distance=0.5mm]below left:$i_k$}] at (ik)   {};
    \node[label={[label distance=0.5mm]below right:$i_{k+1}$}] at (ikp1) {};
    % j 点
    \node[label={[label distance=0.5mm]above left:$j_{m-1}$}]  at (jm1)  {};
    \node[label={[label distance=0.5mm]above right:$j_0$}]  at (j0)  {};
    \node[label={[label distance=0.5mm]above right:$j_1$}]    at (j1)   {};
    \node[label={[label distance=0.5mm]above left:$j_{k-1}$}] at (jkm1) {};
    \node[label={[label distance=0.5mm]above left:$j_k$}]    at (jk)   {};
    \node[label={[label distance=0.5mm]above right:$j_{k+1}$}] at (jkp1) {};
    % a, b
    \node[label={[label distance=0.5mm]above:$a$}] at (A) {};
    \node[label={[label distance=0.5mm]below:$b$}] at (B) {};
    
    % ========== 箭头（全部使用直线，并加粗） ==========
    % a → 所有 i
    \foreach \i in {im1,i0,i1,ikm1,ik,ikp1}
        \draw[->, thick] (A) -- (\i);
    
    % 所有 j → b
    \foreach \j in {jm1,j0,j1,jkm1,jk,jkp1}
        \draw[->, thick] (\j) -- (B);
    
    % i → 同下标 j
    \draw[->, thick] (im1) -- (jm1);
    \draw[->, thick] (i0)  -- (j0);
    \draw[->, thick] (i1)  -- (j1);
    \draw[->, thick] (ikm1)-- (jkm1);
    \draw[->, thick] (ik)  -- (jk);
    \draw[->, thick] (ikp1)-- (jkp1);
    
    % 交叉边
    \draw[->, thick] (i0)  -- (jm1);
    \draw[->, thick] (i1)  -- (j0);
    \draw[->, thick] (ik)  -- (jkm1);
    \draw[->, thick] (ikp1) -- (jk);
    \end{tikzpicture}
    \caption{Schematic diagram of $T_m$}
    \label{fig:T_m}
\end{figure}
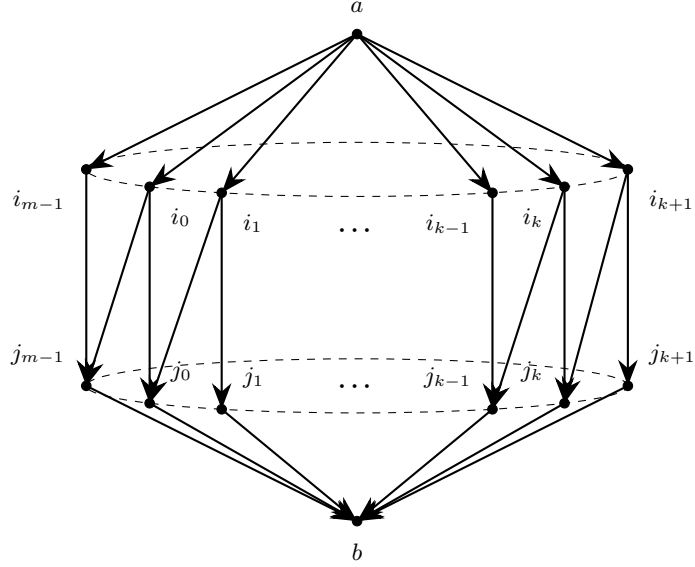

\begin{proposition}\label{pro:trapezohedron}
For the trapezohedron $T_m$, define 
$$\tau_m = \sum_{k=0}^{m-1} \bigl( e_{a i_k j_k b} - e_{a i_{k+1} j_k b} \bigr).$$
Then we have $\Omega_3(T_m) = \langle \tau_m \rangle$.
\end{proposition}

\begin{proof}
All allowed 3-paths in $T_m$ are as $e_{a i_k j_k b},  e_{a i_{k+1} j_k b},$ for all $k = 0, \dots, m - 1$(note that $i_m=i_0$). Let us find all linear combinations of these paths that are $\partial$-invariant. Consider such a linear combination
$$\omega = \sum_{k=0}^{m-1} \bigl( \alpha_k e_{a i_k j_k b} + \beta_k e_{a i_{k+1} j_k b} \bigr)$$
with coefficients $\alpha_k, \beta_k$, and assume that $\omega$ is $\partial$-invariant. We have

$$\begin{aligned}
\partial \omega =& \sum_{k=0}^{m-1} \partial \bigl( \alpha_k e_{a i_k j_k b} + \beta_k e_{a i_{k+1} j_k b} \bigr) \\
=& \sum_{k=0}^{m-1} \bigl( \alpha_k e_{i_k j_k b} + \beta_k e_{i_{k+1} j_k b} \bigr) - \sum_{k=0}^{m-1} \bigl( \alpha_k e_{a j_k b} + \beta_k e_{a j_k b} \bigr) \\
&+ \sum_{k=0}^{m-1} \bigl( \alpha_k e_{a i_k b} + \beta_k e_{a i_{k+1} b} \bigr) - \sum_{k=0}^{m-1} \bigl( \alpha_k e_{a i_k j_k} + \beta_k e_{a i_{k+1} j_k} \bigr).
\end{aligned}$$

The path $e_{a j_k b}$ is not allowed and, hence, must cancel out, which yields $\alpha_k = -\beta_k$. Therefore,
$$\sum_{k=0}^{m-1} \bigl( \alpha_k e_{a i_k b} + \beta_k e_{a i_{k+1} b} \bigr) = \sum_{k=0}^{m-1} \bigl( \alpha_k e_{a i_k b} - \alpha_k e_{a i_{k+1} b} \bigr) = \sum_{k=0}^{m-1} (\alpha_k - \alpha_{k-1}) e_{a i_k b},$$
and it must vanish as $e_{a i_k b}$ is not allowed, whence $\alpha_k = \alpha_{k-1}$. Setting $\alpha_k \equiv \alpha$ and, hence, $\beta_k = -\alpha$, we obtain that
$$\omega = \alpha \sum_{k=0}^{m-1} \bigl( e_{a i_k j_k b} - e_{a i_{k+1} j_k b} \bigr) = \alpha \tau_m$$
so that $\Omega_3(T_m) = \langle \tau_m \rangle$.
\end{proof}

\begin{comment}
\subsection{Graph}

\begin{definition}
A \textit{graph} is a pair $G=(V,E)$ of sets such that $E \subset [V]^2$; thus, the elements of $E$ are 2-element subsets of $V$. The elements of $V$ are the \textit{vertices} of the graph $G$, the elements of $E$ are its \textit{edges}.
\end{definition}

\begin{definition}
A \textit{path} is a non-empty graph $P = (V, E)$ of the form
$$V = \{x_0, x_1, \dots, x_k\}, \quad E = \{x_0x_1, x_1x_2, \dots, x_{k-1}x_k\},$$
where the $x_i$ are all distinct.
\end{definition}

We often refer to a path by the natural sequence of its vertices, writing, say, $P = x_0 x_1 \dots x_k$. 

If $P = x_0 \dots x_k$ is a path and $k \geq 3$, then the graph $C := P + x_k x_0$ is called a \textit{cycle}. As with paths, we often denote a cycle by its sequence of vertices; the above cycle $C$ might be written as $x_0 \dots x_k x_0$. 
\end{comment}

\section{Structure of $\Omega_3$}\label{sec:structureomega3}

This section generalizes \cite[Theorem 2.10]{grigoryan2022advances}. By a new argument, we remove the original assumptions that $G$ contains neither multisquares nor double arrows, thereby establishing the conclusion in a more general setting. We first classify the merging images of trapezohedrons in Section \ref{subsec1:structureomega3}, and then complete the proof of the main theorem in Section \ref{subsec2:structureomega3}.%这里还暂时没有写具体的创新点在哪里

\subsection{The merging images of trapezohedron}\label{subsec1:structureomega3}

\begin{lemma}\label{lem:image1}
Let $G$ be a directed graph. 
Suppose that the vertices $a, i, j, b$ (not necessarily pairwise distinct) satisfy
$$a \to i \to j \to b,\quad i \to b \; or \; i=b,\quad a \to j \; or \; a=j. $$
Then $\omega= e_{a i j b}$ satisfies $\omega \in \Omega_3$, and $\omega$ is a merging image of $\tau_2$.
\end{lemma}

\begin{figure}[htbp]
    \centering
    \begin{tikzpicture}[>= {Stealth[scale=1.5]}, scale=1.2, font=\small]

    % ------------------ 左图 ------------------
    \begin{scope}[local bounding box=left] % 为左图创建边界框，便于画箭头
    \coordinate (A)  at (0,2);    % a'
    \coordinate (i0) at (-1.2,1); % i'_0
    \coordinate (i1) at (1.2,1);  % i'_1
    \coordinate (j0) at (-1.2,-1);% j'_0
    \coordinate (j1) at (1.2,-1); % j'_1
    \coordinate (B)  at (0,-2);   % b'

    % ========== 绘制等大实心黑点 ==========
    \foreach \p in {A,i0,i1,j0,j1,B}
        \draw[fill=black] (\p) circle (2pt);

    % ========== 顶点标签（位置优化，避免重叠）==========
    \node[label={[label distance=0.5mm]above:{$a'$}}]      at (A)  {};
    \node[label={[label distance=0.5mm]left:{$i'_0$}}] at (i0) {};
    \node[label={[label distance=0.5mm]right:{$i'_1$}}]  at (i1) {};
    \node[label={[label distance=0.5mm]left:{$j'_0$}}] at (j0) {};
    \node[label={[label distance=0.5mm]right:{$j'_1$}}]  at (j1) {};
    \node[label={[label distance=0.5mm]below:{$b'$}}]      at (B)  {};

    % ========== 所有有向边（严格遵循给定关系）==========
    \draw[->, thick] (A)  -- (i0);
    \draw[->, thick] (A)  -- (i1);
    \draw[->, thick] (i0) -- (j0);
    \draw[->, thick] (i1) -- (j1);
    \draw[->, thick] (j0) -- (B);
    \draw[->, thick] (j1) -- (B);
    \draw[->, thick] (i1) -- (j0);   % 交叉边
    \draw[->, thick] (i0) -- (j1);   % 交叉边
    \end{scope}
    
    % ------------------ 右图（整体向右平移）------------------
    \begin{scope}[xshift=5cm, local bounding box=right]
    \coordinate (A)  at (0,2);    % a'
    \coordinate (i0) at (-1.2,1); % i'_0
    \coordinate (i1) at (1.2,1);  % i'_1
    \coordinate (j0) at (-1.2,-1);% j'_0
    \coordinate (j1) at (1.2,-1); % j'_1
    \coordinate (B)  at (0,-2);   % b'

    % ========== 绘制等大实心黑点 ==========
    \foreach \p in {A,i0,j0,B}
        \draw[fill=black] (\p) circle (2pt);

    % ========== 顶点标签（位置优化，避免重叠）==========
    \node[label={[label distance=0.5mm]above:{$a$}}]      at (A)  {};
    \node[label={[label distance=0.5mm]left:{$i$}}] at (i0) {};
    \node[label={[label distance=0.5mm]left:{$j$}}] at (j0) {};
    \node[label={[label distance=0.5mm]below:{$b$}}]      at (B)  {};
    
    % ========== 所有有向边（严格遵循给定关系）==========
    \draw[->, thick] (A)  -- (i0);
    \draw[->, thick] (A)  -- (j0);
    \draw[->, thick] (i0) -- (j0);
    \draw[->, thick] (j0) -- (B);
    \draw[->, thick] (i0) -- (B); 
    \end{scope}
    
    % ========== 中间大箭头：从左图指向右图 ==========
    \draw[->, ultra thick, line width=1.5pt] 
        ([xshift=+0.2cm]left.east)   % 从左图右侧边界稍右一点出发
        -- 
        ([xshift=-0.2cm]right.west) % 指向右图左侧边界稍左一点
        node[midway, above, font=\large] {$f$}; % 可加标签，也可不加
    \end{tikzpicture}
    \caption{Schematic diagram of $f : T_2 \to G$ in Lemma\ref{lem:image1}}
    \label{fig:image1}
\end{figure}
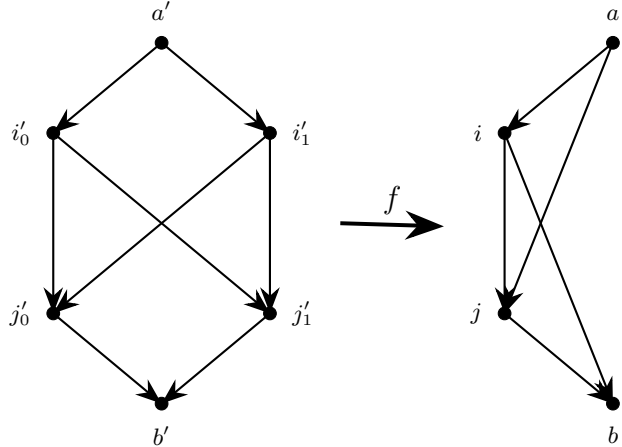

\begin{proof}
Let the vertices of the trapezohedron $T_2$ be $a', b', i'_0, i'_1, j'_0, j'_1$ satisfying 
$$a' \to i'_0 \to j'_0 \to b', \quad a' \to i'_1 \to j'_1 \to b', \quad i'_1 \to j'_0, \quad i'_0 \to j'_1.$$ 
Then 
$$\tau_2 = e_{a' i'_0 j'_0 b'} - e_{a' i'_1 j'_0 b'} +e_{a' i'_1 j'_1 b'} - e_{a' i'_0 j'_1 b'} .$$
Define a morphism $f : T_2 \to G$ by
$$f(a')=a, f(b')=b, f(i'_0)=i, f(j'_0)=j, f(i'_1)=j, f(j'_1)=b.$$
(cf. Fig.\ref{fig:image1}) Let $f_* : \mathcal{A}_3(T_2) \to \mathcal{A}_3(G)$ be the induced mapping, we have
$$\begin{aligned}
f_*(\tau_2) &= e_{a i j b} - e_{a j j b} +e_{a j b b} - e_{a i b b} = e_{a i j b} =\omega.
\end{aligned}$$
We have $\omega$ is a merging image of $\tau_2$ since $f(b') = f(j'_1) = b$. By Proposition \ref{pro:morphism}, it follows that $\omega \in \Omega_3$.
\end{proof}

\iffalse
Since the boundary of $\omega$ is
$$\partial e_{a i j b}
= e_{i j b} - e_{a j b} + e_{a i b} - e_{a i j} \in \mathcal{A}_2.$$
Hence $\omega \in \Omega_3$.
\fi

\begin{lemma}\label{lem:image0}
Let $G$ be a directed graph. 
Suppose that the vertices $a,\ \{i_k\}_{k=0}^{m-1},\ \{j_k\}_{k=0}^{m-1},\ b$ (not necessarily pairwise distinct) satisfy
$$a \to i_k \to j_k \to b,\quad i_{k+1} \to j_k,$$
for all $k = 0,1,\dots,m-1$, where the indices are taken modulo m.
Then
$$\omega= e_{a i_0 j_0 b} - e_{a i_1 j_0 b} + e_{a i_1 j_1 b} - \cdots + e_{a i_{m-1} j_{m-1} b} - e_{a i_0 j_{m-1} b}$$
satisfies $\omega \in \Omega_3$, and $\omega$ is $\tau_{m}$ or a merging image of $\tau_{m}$.
\end{lemma}

\begin{proof}
If $a,\ \{i_k\}_{k=0}^{m-1},\ \{j_k\}_{k=0}^{m-1},\ b$ are pairwise distinct, then $\omega$ is $\tau_m$. If $a,\ \{i_k\}_{k=0}^{m-1},\ \{j_k\}_{k=0}^{m-1},\ b$ are not pairwise distinct, let the vertices of the trapezohedron $T_m$ be $a', b', i'_0, \dots, i'_{m-1}, j'_0,  \dots, j'_{m-1}$
satisfying 
$$a' \to i'_k \to j'_k \to b', \quad i'_{k+1} \to j'_k$$ 
for all $k = 0, \dots, m-1 \pmod m$.
Then 
$$\tau_m = \sum_{k=0}^{m-1} \bigl( e_{a' i'_k j'_k b'} - e_{a' i'_{k+1} j'_k b'} \bigr).$$
Define a morphism $f : T_m \to G$ by
$$f(a')=a, f(b')=b, f(i'_k)=i_k, f(j'_k)=j_k$$
for all $k = 0, \dots, m-1$. Let $f_* : \mathcal{A}_3(T_m) \to \mathcal{A}_3(G)$ be the induced mapping, we have
$$f_*(\tau_m) = \sum_{k=0}^{m-1} \bigl( e_{a i_k j_k b} - e_{a i_{k+1} j_k b} \bigr) =\omega.$$
We have $\omega$ is a merging image of $\tau_m$ since $a,\ \{i_k\}_{k=0}^{m-1},\ \{j_k\}_{k=0}^{m-1},\ b$ are not pairwise distinct. By Proposition \ref{pro:morphism}, it follows that $\omega \in \Omega_3$.
\end{proof}

\iffalse
The boundary of $\omega$ is
$$\begin{aligned}
\partial \omega &= \sum_{k=0}^{m-1} (\partial e_{a i_k j_k b} - \partial e_{a i_{k+1} j_k b}) \\
&= \sum_{k=0}^{m-1} (e_{i_k j_k b} - e_{a j_k b} + e_{a i_k b} - e_{a i_k j_k} - e_{i_{k+1} j_k b} + e_{a j_k b} - e_{a i_{k+1} b} + e_{a i_{k+1} j_k} ) \\
&= \sum_{k=0}^{m-1} ( e_{i_k j_k b} - e_{a i_k j_k} - e_{i_{k+1} j_k b} + e_{a i_{k+1} j_k} ) \in \mathcal{A}_2.
\end{aligned}$$
Therefore, $\omega \in \Omega_3$.
\fi

\begin{lemma}\label{lem:image2}
Let $G$ be a directed graph. 
Suppose that the vertices $a,\ \{i_k\}_{k=0}^{m+1},\ \{j_k\}_{k=0}^m,\ b$ (not necessarily pairwise distinct) satisfy
$$a \to i_k \to j_k \to b, \quad i_{k+1} \to j_k,\quad i_0 \to b \; or \; i_0 = b,\quad i_{m+1} \to b \; or \; i_{m+1} = b,\quad a \to i_{m+1},$$
for all $k = 0,1,\dots,m$.
Then
$$\omega= e_{a i_0 j_0 b} - e_{a i_1 j_0 b} + e_{a i_1 j_1 b} - \cdots + e_{a i_m j_m b} - e_{a i_{m+1} j_m b}$$
satisfies $\omega \in \Omega_3$, and $\omega$ is a merging image of $\tau_{m+2}$.
\end{lemma}

\begin{figure}[htbp]
    \centering
    \begin{tikzpicture}[>= {Stealth[scale=1.5]}, scale=0.7, font=\small]
    % ========== 全局参数（左右图共用） ==========
    \def\xR{4}
    \def\yR{0.4}
    \def\jY{-3.2}
    
    % ------------------ 左图 ------------------
    \begin{scope}[local bounding box=left] % 为左图创建边界框，便于画箭头
        % 虚线椭圆
        \draw[dashed] (0,0) ellipse [x radius=\xR, y radius=\yR];
        \draw[dashed] (0,\jY) ellipse [x radius=\xR, y radius=\yR];
        
        % 节点坐标（与原来完全相同）
        \coordinate (im1)  at ({\xR*cos(180)}, {\yR*sin(180)});
        \coordinate (i0)   at ({\xR*cos(220)}, {\yR*sin(220)});
        \coordinate (i1)   at ({\xR*cos(240)}, {\yR*sin(240)});
        \coordinate (ikm1) at ({\xR*cos(300)}, {\yR*sin(300)});
        \coordinate (ik)   at ({\xR*cos(320)}, {\yR*sin(320)});
        \coordinate (ikp1) at ({\xR*cos(360)}, {\yR*sin(360)});
        
        \coordinate (jm1)  at ({\xR*cos(180)}, {\jY+\yR*sin(180)});
        \coordinate (j0)   at ({\xR*cos(220)}, {\jY+\yR*sin(220)});
        \coordinate (j1)   at ({\xR*cos(240)}, {\jY+\yR*sin(240)});
        \coordinate (jkm1) at ({\xR*cos(300)}, {\jY+\yR*sin(300)});
        \coordinate (jk)   at ({\xR*cos(320)}, {\jY+\yR*sin(320)});
        \coordinate (jkp1) at ({\xR*cos(360)}, {\jY+\yR*sin(360)});
        
        \coordinate (A) at (0,2);
        \coordinate (B) at (0,\jY-2);
        
        % 小黑点
        \foreach \p in {im1,i0,i1,ikm1,ik,ikp1, jm1,j0,j1,jkm1,jk,jkp1, A, B}
            \draw[fill=black] (\p) circle (2pt);
        
        % 省略号
        \node at ({\xR*cos(270)}, {\yR*sin(270)-0.5}) {$\boldsymbol{\cdots}$};
        \node at ({\xR*cos(270)}, {\jY+\yR*sin(270)+0.4}) {$\boldsymbol{\cdots}$};
        
        % 标签
        \node[label={[label distance=0.5mm]below left:$i'_{m}$}] at (im1) {};
        \node[label={[label distance=0.5mm]below:$i'_{m+1}$}] at (i0)  {};
        \node[label={[label distance=0.5mm]below right:$i'_0$}] at (i1)  {};
        \node[label={[label distance=0.5mm]below left:$i'_{k-1}$}] at (ikm1) {};
        \node[label={[label distance=0.5mm]below left:$i'_k$}] at (ik)   {};
        \node[label={[label distance=0.5mm]below right:$i'_{k+1}$}] at (ikp1) {};
        \node[label={[label distance=0.5mm]above left:$j'_{m}$}]  at (jm1)  {};
        \node[label={[label distance=1.2mm]above:$j'_{m+1}$}]  at (j0)  {};
        \node[label={[label distance=0.5mm]above right:$j'_0$}]    at (j1)   {};
        \node[label={[label distance=0.5mm]above left:$j'_{k-1}$}] at (jkm1) {};
        \node[label={[label distance=0.5mm]above left:$j'_k$}]    at (jk)   {};
        \node[label={[label distance=0.5mm]above right:$j'_{k+1}$}] at (jkp1) {};
        \node[label={[label distance=0.5mm]above:$a'$}] at (A) {};
        \node[label={[label distance=0.5mm]below:$b'$}] at (B) {};
        
        % 箭头（直线）
        \foreach \i in {im1,i0,i1,ikm1,ik,ikp1} \draw[->, thick] (A) -- (\i);
        \foreach \j in {jm1,j0,j1,jkm1,jk,jkp1} \draw[->, thick] (\j) -- (B);
        \draw[->, thick] (im1)--(jm1); \draw[->, thick] (i0)--(j0);
        \draw[->, thick] (i1)--(j1);   \draw[->, thick] (ikm1)--(jkm1);
        \draw[->, thick] (ik)--(jk);   \draw[->, thick] (ikp1)--(jkp1);
        \draw[->, thick] (i0)--(jm1);  \draw[->, thick] (i1)--(j0);
        \draw[->, thick] (ik)--(jkm1); \draw[->, thick] (ikp1)--(jk);
    \end{scope}
    
    % ------------------ 右图（整体向右平移12cm）------------------
    \begin{scope}[xshift=12cm, local bounding box=right] % 右移12cm
        % 内容与左图完全一致（直接复制）
        \draw[dashed] (0,0) ellipse [x radius=\xR, y radius=\yR];
        \draw[dashed] (0,\jY) ellipse [x radius=\xR, y radius=\yR];
        
        \coordinate (im1)  at ({\xR*cos(180)}, {\yR*sin(180)});
        \coordinate (i0)   at ({\xR*cos(220)}, {\yR*sin(220)});
        \coordinate (i1)   at ({\xR*cos(240)}, {\yR*sin(240)});
        \coordinate (ikm1) at ({\xR*cos(300)}, {\yR*sin(300)});
        \coordinate (ik)   at ({\xR*cos(320)}, {\yR*sin(320)});
        \coordinate (ikp1) at ({\xR*cos(360)}, {\yR*sin(360)});
        
        \coordinate (jm1)  at ({\xR*cos(180)}, {\jY+\yR*sin(180)});
        \coordinate (j0)   at ({\xR*cos(220)}, {\jY+\yR*sin(220)});
        \coordinate (j1)   at ({\xR*cos(240)}, {\jY+\yR*sin(240)});
        \coordinate (jkm1) at ({\xR*cos(300)}, {\jY+\yR*sin(300)});
        \coordinate (jk)   at ({\xR*cos(320)}, {\jY+\yR*sin(320)});
        \coordinate (jkp1) at ({\xR*cos(360)}, {\jY+\yR*sin(360)});
        
        \coordinate (A) at (0,2);
        \coordinate (B) at (0,\jY-2);
        
        \foreach \p in {im1,i0,i1,ikm1,ik,ikp1, jm1,j1,jkm1,jk,jkp1, A, B}
            \draw[fill=black] (\p) circle (2pt);
        
        \node at ({\xR*cos(270)}, {\yR*sin(270)-0.5}) {$\boldsymbol{\cdots}$};
        \node at ({\xR*cos(270)}, {\jY+\yR*sin(270)+0.4}) {$\boldsymbol{\cdots}$};
        
        \node[label={[label distance=0.5mm]below left:$i_{m}$}] at (im1) {};
        \node[label={[label distance=1.2mm]below:$i_{m+1}$}] at (i0)  {};
        \node[label={[label distance=0.5mm]below right:$i_0$}] at (i1)  {};
        \node[label={[label distance=0.5mm]below left:$i_{k-1}$}] at (ikm1) {};
        \node[label={[label distance=0.5mm]below left:$i_k$}] at (ik)   {};
        \node[label={[label distance=0.5mm]below right:$i_{k+1}$}] at (ikp1) {};
        \node[label={[label distance=0.5mm]above left:$j_{m}$}]  at (jm1)  {};
        \node[label={[label distance=0.5mm]above left:$j_0$}]    at (j1)   {};
        \node[label={[label distance=0.5mm]above left:$j_{k-1}$}] at (jkm1) {};
        \node[label={[label distance=0.5mm]above left:$j_k$}]    at (jk)   {};
        \node[label={[label distance=0.5mm]above right:$j_{k+1}$}] at (jkp1) {};
        \node[label={[label distance=0.5mm]above:$a$}] at (A) {};
        \node[label={[label distance=0.5mm]below:$b$}] at (B) {};
        
        \foreach \i in {im1,i0,i1,ikm1,ik,ikp1} \draw[->, thick] (A) -- (\i);
        \foreach \j in {jm1,j1,jkm1,jk,jkp1} \draw[->, thick] (\j) -- (B);
        \draw[->, thick] (im1)--(jm1); \draw[->, thick] (i0)--(B);
        \draw[->, thick] (i1)--(j1);   \draw[->, thick] (ikm1)--(jkm1);
        \draw[->, thick] (ik)--(jk);   \draw[->, thick] (ikp1)--(jkp1);
        \draw[->, thick] (i0)--(jm1);  \draw[->, thick] (i1)--(B);
        \draw[->, thick] (ik)--(jkm1); \draw[->, thick] (ikp1)--(jk);
    \end{scope}
    
    % ========== 中间大箭头：从左图指向右图 ==========
    \draw[->, ultra thick, line width=1.5pt] 
        ([xshift=-0.4cm]left.east)   % 从左图右侧边界稍右一点出发
        -- 
        ([xshift=0.4cm]right.west) % 指向右图左侧边界稍左一点
        node[midway, above, font=\large] {$f$}; % 可加标签，也可不加
    \end{tikzpicture}
    \caption{Schematic diagram of $f : T_{m+2} \to G$ in Lemma\ref{lem:image2}}
    \label{fig:image2}
\end{figure}

\begin{proof}
Let the vertices of the trapezohedron $T_{m+2}$ be $a', b', i'_0, \dots, i'_{m+1}, j'_0,  \dots, j'_{m+1}$
satisfying 
$$a' \to i'_k \to j'_k \to b', \quad i'_{k+1} \to j'_k$$ 
for all $k = 0, \dots, m+1 \pmod{m+2}$.
Then 
$$\tau_{m+2} = \sum_{k=0}^{m+1} \bigl( e_{a' i'_k j'_k b'} - e_{a' i'_{k+1} j'_k b'} \bigr).$$
Define a morphism $f : T_{m+2} \to G$ by
$$f(a')=a, f(b')=b, f(i'_k)=i_k, f(j'_k)=j_k, f(i'_{m+1})=i_{m+1}, f(j'_{m+1})=b$$
for all $k = 0, \dots, m$. (cf. Fig.\ref{fig:image2}) Let $f_* : \mathcal{A}_3(T_{m+2}) \to \mathcal{A}_3(G)$ be the induced mapping, we have
$$\begin{aligned}
f_*(\tau_{m+2}) &= \sum_{k=0}^m \bigl( e_{a i_k j_k b} - e_{a i_{k+1} j_k b} \bigr) + e_{a i_{m+1} b b} - e_{a i_0 b b}\\
&=\sum_{k=0}^m \bigl( e_{a i_k j_k b} - e_{a i_{k+1} j_k b} \bigr) =\omega.
\end{aligned}$$
We have $\omega$ is a merging image of $\tau_{m+2}$ since $f(b') = f(j'_{m+1}) = b$. Since $\tau_{m+2}\in \Omega_3(T_{m+2})$, by Proposition \ref{pro:morphism}, it follows that $\omega \in \Omega_3$.
\end{proof}

\iffalse
The boundary of $\omega$ is
$$\begin{aligned}
\partial \omega 
&= \sum_{k=0}^{m} (\partial e_{a i_k j_k b} - \partial e_{a i_{k+1} j_k b})\\
&= \sum_{k=0}^{m} (e_{i_k j_k b} - e_{a j_k b} + e_{a i_k b} - e_{a i_k j_k} - e_{i_{k+1} j_k b} + e_{a j_k b} - e_{a i_{k+1} b} + e_{a i_{k+1} j_k} ) \\
&= \sum_{k=0}^{m} ( e_{i_k j_k b} - e_{a i_k j_k} + e_{i_{k+1} j_k b} - e_{a i_{k+1} j_k} )+ e_{a i_0 b} - e_{a i_{m+1} b} \\
&\in \mathcal{A}_2.
\end{aligned}$$
Therefore, $\omega \in \Omega_3$.
\fi

\begin{lemma}\label{lem:image3}
Let $G$ be a directed graph. 
Suppose that the vertices $a,\ \{i_k\}_{k=0}^m,\ \{j_k\}_{k=0}^m,\ b$ (not necessarily pairwise distinct) satisfy
$$a \to i_k \to j_k \to b,\quad a \to i_m \to j_m \to b,\quad i_{k+1} \to j_k,\quad i_0 \to b \; or \; i_0 = b,\quad a \to j_m \; or \;a=j_m,$$
for all $k = 0,1,\dots,m-1$.
Then
$$\omega= e_{a i_0 j_0 b} - e_{a i_1 j_0 b} + e_{a i_1 j_1 b} - \cdots - e_{a i_m j_{m-1} b} + e_{a i_m j_m b}$$
satisfies $\omega \in \Omega_3$, and $\omega$ is a merging image of $\tau_{m+2}$.
\end{lemma}

\begin{figure}[htbp]
    \centering
    \begin{tikzpicture}[>= {Stealth[scale=1.5]}, scale=0.7, font=\small]
    % ========== 全局参数（左右图共用） ==========
    \def\xR{4}
    \def\yR{0.4}
    \def\jY{-3.2}
    
    % ------------------ 左图 ------------------
    \begin{scope}[local bounding box=left] % 为左图创建边界框，便于画箭头
        % 虚线椭圆
        \draw[dashed] (0,0) ellipse [x radius=\xR, y radius=\yR];
        \draw[dashed] (0,\jY) ellipse [x radius=\xR, y radius=\yR];
        
        % 节点坐标（与原来完全相同）
        \coordinate (im1)  at ({\xR*cos(180)}, {\yR*sin(180)});
        \coordinate (i0)   at ({\xR*cos(220)}, {\yR*sin(220)});
        \coordinate (i1)   at ({\xR*cos(240)}, {\yR*sin(240)});
        \coordinate (ikm1) at ({\xR*cos(300)}, {\yR*sin(300)});
        \coordinate (ik)   at ({\xR*cos(320)}, {\yR*sin(320)});
        \coordinate (ikp1) at ({\xR*cos(360)}, {\yR*sin(360)});
        
        \coordinate (jm1)  at ({\xR*cos(180)}, {\jY+\yR*sin(180)});
        \coordinate (j0)   at ({\xR*cos(220)}, {\jY+\yR*sin(220)});
        \coordinate (j1)   at ({\xR*cos(240)}, {\jY+\yR*sin(240)});
        \coordinate (jkm1) at ({\xR*cos(300)}, {\jY+\yR*sin(300)});
        \coordinate (jk)   at ({\xR*cos(320)}, {\jY+\yR*sin(320)});
        \coordinate (jkp1) at ({\xR*cos(360)}, {\jY+\yR*sin(360)});
        
        \coordinate (A) at (0,2);
        \coordinate (B) at (0,\jY-2);
        
        % 小黑点
        \foreach \p in {im1,i0,i1,ikm1,ik,ikp1, jm1,j0,j1,jkm1,jk,jkp1, A, B}
            \draw[fill=black] (\p) circle (2pt);
        
        % 省略号
        \node at ({\xR*cos(270)}, {\yR*sin(270)-0.5}) {$\boldsymbol{\cdots}$};
        \node at ({\xR*cos(270)}, {\jY+\yR*sin(270)+0.4}) {$\boldsymbol{\cdots}$};
        
        % 标签
        \node[label={[label distance=0.5mm]below left:$i'_{m}$}] at (im1) {};
        \node[label={[label distance=0.5mm]below:$i'_{m+1}$}] at (i0)  {};
        \node[label={[label distance=0.5mm]below right:$i'_0$}] at (i1)  {};
        \node[label={[label distance=0.5mm]below left:$i'_{k-1}$}] at (ikm1) {};
        \node[label={[label distance=0.5mm]below left:$i'_k$}] at (ik)   {};
        \node[label={[label distance=0.5mm]below right:$i'_{k+1}$}] at (ikp1) {};
        \node[label={[label distance=0.5mm]above left:$j'_{m}$}]  at (jm1)  {};
        \node[label={[label distance=1.2mm]above:$j'_{m+1}$}]  at (j0)  {};
        \node[label={[label distance=0.5mm]above right:$j'_0$}]    at (j1)   {};
        \node[label={[label distance=0.5mm]above left:$j'_{k-1}$}] at (jkm1) {};
        \node[label={[label distance=0.5mm]above left:$j'_k$}]    at (jk)   {};
        \node[label={[label distance=0.5mm]above right:$j'_{k+1}$}] at (jkp1) {};
        \node[label={[label distance=0.5mm]above:$a'$}] at (A) {};
        \node[label={[label distance=0.5mm]below:$b'$}] at (B) {};
        
        % 箭头（直线）
        \foreach \i in {im1,i0,i1,ikm1,ik,ikp1} \draw[->, thick] (A) -- (\i);
        \foreach \j in {jm1,j0,j1,jkm1,jk,jkp1} \draw[->, thick] (\j) -- (B);
        \draw[->, thick] (im1)--(jm1); \draw[->, thick] (i0)--(j0);
        \draw[->, thick] (i1)--(j1);   \draw[->, thick] (ikm1)--(jkm1);
        \draw[->, thick] (ik)--(jk);   \draw[->, thick] (ikp1)--(jkp1);
        \draw[->, thick] (i0)--(jm1);  \draw[->, thick] (i1)--(j0);
        \draw[->, thick] (ik)--(jkm1); \draw[->, thick] (ikp1)--(jk);
    \end{scope}
    
    % ------------------ 右图（整体向右平移12cm）------------------
    \begin{scope}[xshift=12cm, local bounding box=right] % 右移12cm
        % 内容与左图完全一致（直接复制）
        \draw[dashed] (0,0) ellipse [x radius=\xR, y radius=\yR];
        \draw[dashed] (0,\jY) ellipse [x radius=\xR, y radius=\yR];
        
        \coordinate (im1)  at ({\xR*cos(180)}, {\yR*sin(180)});
        \coordinate (i0)   at ({\xR*cos(220)}, {\yR*sin(220)});
        \coordinate (i1)   at ({\xR*cos(230)}, {\yR*sin(230)});
        \coordinate (ikm1) at ({\xR*cos(300)}, {\yR*sin(300)});
        \coordinate (ik)   at ({\xR*cos(320)}, {\yR*sin(320)});
        \coordinate (ikp1) at ({\xR*cos(360)}, {\yR*sin(360)});
        
        \coordinate (jm1)  at ({\xR*cos(180)}, {\jY+\yR*sin(180)});
        \coordinate (j0)   at ({\xR*cos(220)}, {\jY+\yR*sin(220)});
        \coordinate (j1)   at ({\xR*cos(230)}, {\jY+\yR*sin(230)});
        \coordinate (jkm1) at ({\xR*cos(300)}, {\jY+\yR*sin(300)});
        \coordinate (jk)   at ({\xR*cos(320)}, {\jY+\yR*sin(320)});
        \coordinate (jkp1) at ({\xR*cos(360)}, {\jY+\yR*sin(360)});
        
        \coordinate (A) at (0,2);
        \coordinate (B) at (0,\jY-2);
        
        \foreach \p in {im1,i1,ikm1,ik,ikp1, jm1,j1,jkm1,jk,jkp1, A, B}
            \draw[fill=black] (\p) circle (2pt);
        
        \node at ({\xR*cos(270)}, {\yR*sin(270)-0.5}) {$\boldsymbol{\cdots}$};
        \node at ({\xR*cos(270)}, {\jY+\yR*sin(270)+0.4}) {$\boldsymbol{\cdots}$};
        
        \node[label={[label distance=0.5mm]below left:$i_{m}$}] at (im1) {};
        \node[label={[label distance=0.5mm]below left:$i_0$}] at (i1)  {};
        \node[label={[label distance=0.5mm]below left:$i_{k-1}$}] at (ikm1) {};
        \node[label={[label distance=0.5mm]below left:$i_k$}] at (ik)   {};
        \node[label={[label distance=0.5mm]below right:$i_{k+1}$}] at (ikp1) {};
        \node[label={[label distance=0.5mm]above left:$j_{m}$}]  at (jm1)  {};
        \node[label={[label distance=0.5mm]above left:$j_0$}]    at (j1)   {};
        \node[label={[label distance=0.5mm]above left:$j_{k-1}$}] at (jkm1) {};
        \node[label={[label distance=0.5mm]above left:$j_k$}]    at (jk)   {};
        \node[label={[label distance=0.5mm]above right:$j_{k+1}$}] at (jkp1) {};
        \node[label={[label distance=0.5mm]above:$a$}] at (A) {};
        \node[label={[label distance=0.5mm]below:$b$}] at (B) {};
        
        \foreach \i in {im1,i1,ikm1,ik,ikp1} \draw[->, thick] (A) -- (\i);
        \foreach \j in {jm1,i1,j1,jkm1,jk,jkp1} \draw[->, thick] (\j) -- (B);
        \draw[->, thick] (im1)--(jm1); \draw[->, thick] (A)--(i1);
        \draw[->, thick] (i1)--(j1);   \draw[->, thick] (ikm1)--(jkm1);
        \draw[->, thick] (ik)--(jk);   \draw[->, thick] (ikp1)--(jkp1);
        \draw[->, thick] (A)--(jm1);
        \draw[->, thick] (ik)--(jkm1); \draw[->, thick] (ikp1)--(jk);
    \end{scope}
    
    % ========== 中间大箭头：从左图指向右图 ==========
    \draw[->, ultra thick, line width=1.5pt] 
        ([xshift=-0.4cm]left.east)   % 从左图右侧边界稍右一点出发
        -- 
        ([xshift=0.4cm]right.west) % 指向右图左侧边界稍左一点
        node[midway, above, font=\large] {$f$}; % 可加标签，也可不加
    \end{tikzpicture}
    \caption{Schematic diagram of $f : T_{m+2} \to G$ in Lemma\ref{lem:image3}}
    \label{fig:image3}
\end{figure}

\begin{proof}
Let the vertices of the trapezohedron $T_{m+2}$ be $a', b', i'_0, \dots, i'_{m+1}, j'_0,  \dots, j'_{m+1}$
satisfying 
$$a' \to i'_k \to j'_k \to b', \quad i'_{k+1} \to j'_k$$ 
for all $k = 0, \dots, m+1 \pmod{m+2}$.
Then 
$$\tau_{m+2} = \sum_{k=0}^{m+1} \bigl( e_{a' i'_k j'_k b'} - e_{a' i'_{k+1} j'_k b'} \bigr).$$
Define a morphism $f : T_{m+2} \to G$ by
$$f(a')=a, f(b')=b, f(i'_k)=i_k, f(j'_k)=j_k, f(i'_{m+1})=a, f(j'_{m+1})=i_0$$
for all $k = 0, \dots, m$. (cf. Fig.\ref{fig:image3}) Let $f_* : \mathcal{A}_3(T_{m+2}) \to \mathcal{A}_3(G)$ be the induced map, we have
$$\begin{aligned}
f_*(\tau_{m+2}) &= \sum_{k=0}^{m-1} \bigl( e_{a i_k j_k b} - e_{a i_{k+1} j_k b} \bigr) + e_{a i_m j_m b} - e_{a a j_m b} + e_{a a i_0 b} - e_{a i_0 i_0 b}\\
&=\sum_{k=0}^m \bigl( e_{a i_k j_k b} - e_{a i_{k+1} j_k b} \bigr) + e_{a i_m j_m b} = \omega.
\end{aligned}$$
Since $f(a') = f(i'_{m+1}) = a$, $\omega$ is a merging image of $\tau_{m+2}$. By Proposition \ref{pro:morphism}, it follows that $\omega \in \Omega_3$.
\end{proof}

\iffalse
The boundary of $\omega$ is
$$\begin{aligned}
\partial \omega &= \sum_{k=0}^{m-1} (\partial e_{a i_k j_k b} - \partial e_{a i_{k+1} j_k b}) + \partial e_{a i_m j_m b}\\
&= \sum_{k=0}^{m-1} (e_{i_k j_k b} - e_{a j_k b} + e_{a i_k b} - e_{a i_k j_k} - e_{i_{k+1} j_k b} + e_{a j_k b} - e_{a i_{k+1} b} + e_{a i_{k+1} j_k} ) \\
&+ (e_{i_m j_m b} - e_{a j_m b} + e_{a i_m b} - e_{a i_m j_m}) \\
&= \sum_{k=0}^{m-1} ( e_{i_k j_k b} - e_{a i_k j_k} + e_{i_{k+1} j_k b} - e_{a i_{k+1} j_k} ) + ( e_{i_m j_m b} - e_{a i_m j_m}) + e_{a i_0 b} - e_{a j_m b} \\
&\in \mathcal{A}_2.
\end{aligned}$$
Therefore, $\omega \in \Omega_3$.
\fi

\subsection{The structure of $\Omega_3$}\label{subsec2:structureomega3}

To prove the first statement of Theorem \ref{thm:main}, we only need to consider all \emph{minimal} $\partial$-invariant paths. What we need to prove is that all such paths are trapezohedral paths and their merging images. The following theorem shows that the space of paths of $\Omega_3$ is spanned by minimal $\partial$-invariant paths.

\begin{lemma}\label{lem:minimal}
Let G be a digraph. Let a path $\omega \in \Omega_3(G)$ be a minimal $\partial$-invariant path, denote by $P$ the set of all elementary terms $e_{aijb}$ of $\omega$ with non-zero coefficients. If $\omega_0 = \sum_{e \in P} c(e)\, e \ne 0$ satisfies $\omega_0 \in \Omega_3$, where $c(e) \in \mathbb{K}$ for all $e \in P$, then there exists a scalar $c \in \mathbb{K}$ such that $\omega = c\,\omega_0$. 
\end{lemma}

\begin{proof}
Since $\omega_0 \neq 0$, there exists a scalar $c \in \mathbb{K}$ such that $|| \omega - c\omega_0 || < || \omega ||$. This is because, assuming $e_{aijb} \in P$ and $c(e_{aijb}) \neq 0$, we can choose a nonzero constant $c$ such that the coefficient of $e_{aijb}$ in $\omega - c\omega_0$ becomes zero. Consequently, the set of all elementary terms of $\omega - c\omega_0$ with nonzero coefficients is strictly contained in $P$. By the minimality of $\omega$, we must have $\omega - c \omega_0 = 0$.
\end{proof}

\begin{comment}
If $\omega - c \, \omega_0 \ne 0$, there exists a scalar $c' \in \mathbb{K}$ such that $\| \omega - c' \, (\omega - c \, \omega_0) \| < \| \omega \|$. 
Clearly, $\omega - c'\,(\omega - c\,\omega_0) \in \Omega_3$ and $c'\,(\omega - c\,\omega_0) \in \Omega_3$. Hence,
$$\omega = \bigl[\omega - c'\,(\omega - c\,\omega_0)\bigr] + c'\,(\omega - c\,\omega_0)$$
is a sum of other $\partial-$invariant paths with strictly smaller widths,
which contradicts the minimality of $\omega$. 

Hence $\omega - c \, \omega_0 = 0$, that is, $\omega = c\,\omega_0$. 
\end{comment}

Let $G$ be a digraph. By Proposition \ref{pro:cluster}, $\Omega_3(G)$ has a basis that consists of minimal $\partial-$invariant clusters. Let a path $\omega \in \Omega_3$ be a minimal $\partial-$invariant $(a,b)$-cluster.To prove Theorem \ref{thm:main}, it therefore suffices to show that $\omega$ is a merging image of one of the trapezohedral paths $\tau_m$ up to a constant factor.

Denote by $V$ the set of all elementary terms of $\omega$ of the form $e_{a i j b}$. We construct a graph $\Gamma$ with $V=V(\Gamma)$ as the vertex set, and add edges according to the following rules:

\begin{itemize}
\item	We add an edge of color 1 between distinct vertices $e_{a i_1 j b}$ and $e_{a i_2 j b}$ if $a \not\to j$ and $a \ne j$;

\item	We add an edge of color 2 between distinct vertices $e_{a i j_1 b}$ and $e_{a i j_2 b}$ if $i \not\to b$ and $i \ne b$.
\end{itemize}

The edge set is denoted by $E=E(\Gamma)$. We denote the set of edges of color $1$ by
$$E_1 := \{ \, e \in E(\Gamma) \mid e \text{ has color } 1 \, \},$$
and the set of edges of color $2$ by
$$E_2 := \{ \, e \in E(\Gamma) \mid e \text{ has color } 2 \, \}.$$
Clearly $E=E_1\cup E_2$. We then have the following two propositions:

\begin{proposition}\label{clm:exclusion}
For $v, v' \in V$, the relations $vv' \in E_1$ and $vv' \in E_2$ cannot hold simultaneously.
\end{proposition}

\begin{proof}
Assume that both $vv' \in E_1$ and $vv' \in E_2$ hold. Suppose that $v = e_{a i j b}$ and $v' = e_{a i' j' b}$. On the one hand, since $v v' \in E_1$, we must have $j = j'$. On the other hand, since $v v' \in E_2$, we must have $i = i'$, which together imply $v = v'$, a contradiction.
\end{proof}

The second proposition shows that each color class forms a disjoint collection of cliques:
\begin{proposition}\label{clm:transitive}
Let $v_1, v_2, v_3$ be the vertices of the graph $\Gamma$, and let $i \in \{1,2\}$. If $v_1v_2 \in E_i$ and $v_2v_3 \in E_i$, then $v_1v_3 \in E_i$.
\end{proposition}

\begin{proof}
Consider the case $i=1$; the case $i=2$ follows similarly. Suppose $v_1 = e_{a i_1 j b}$. Since $v_1v_2 \in E_1$, there exists $i_2 \in V(G)$ such that $v_2 = e_{a i_2 j b}$. Similarly, from $v_2v_3 \in E_1$, there exists $i_3 \in V(G)$ such that $v_3 = e_{a i_3 j b}$. Finally by the definition of $E_1$, it then follows that $v_1v_3 \in E_1$.
\end{proof}

We say that a path or a cycle in $\Gamma$ is \emph{alternating} if any two adjacent edges on it are assigned distinct colors. We next state a structural property of $\Gamma$.

\begin{proposition}\label{clm:alternating}
Assume that $\Gamma$ contains at least one edge and has no alternating cycle. Let $v_0 v_1 \cdots v_n$ be a longest alternating path in $\Gamma$, and suppose $v_0 = e_{a i j b}$. Then the following hold:
\begin{itemize}
\item If $v_0 v_1 \in E_1$, then either $i \to b$ or $i = b$.
\item If $v_0 v_1 \in E_2$, then either $a \to j$ or $a = j$.
\end{itemize}
\end{proposition}

\begin{proof}
Consider the case $v_0 v_1 \in E_1$; the case $v_0 v_1 \in E_2$ can be proved analogously.

Assume that $i \nrightarrow b$ and $i \ne b$, the term $e_{a i b}$ appearing in $\partial e_{a i j b}$ is non-allowed and must be canceled in $\partial \omega$ by the boundary of another elementary 3-path $v' \in V$ and with endpoints $a,b$. Such a 3-path can only be of the form $e_{a i j' b}$ with $i \to j' \to b$. By definition, $v_0 v' \in E_2$. Since the path $v_0 v_1 \cdots v_n$ is the longest, it follows that $v' \in \{ v_1, v_2, \ldots, v_n \}$. Hence, there exists $t \in \{1,2,\ldots,n\}$ such that $v' = v_t$. 

If $t = 1$, then $v_0 v_1 \in E_1$ and $v_0 v_1 \in E_2$, which contradicts Proposition \ref{clm:exclusion}. Hence $t > 1$. If $v_{t-1} v_t \in E_1$, then $v_0 v_1 \cdots v_{t-1} v_t v_0$ forms an alternating cycle, which leads to a contradiction. If $v_{t-1} v_t \in E_2$, then, since $v_0 v_t \in E_2$, Proposition \ref{clm:transitive} implies that $v_0 v_{t-1} \in E_2$. Consequently, $v_0 v_1 \cdots v_{t-1} v_0$ forms a cycle of Case 2, which leads to a contradiction. Therefore, we must have $i \to b$ or $i = b$.

\end{proof}

Based on the preparation above, we finally prove the first statement of Theorem \ref{thm:main}, which is restated as follows:

\begin{theorem}\label{thm:thm}
Let $G$ be a digraph. Then every minimal $\partial$-invariant path is either a trapezohedron or a merging image thereof, as described in Lemma \ref{lem:image1}, Lemma \ref{lem:image0}, Lemma \ref{lem:image2} or Lemma \ref{lem:image3}. Moreover, $\Omega_3(G)$ admits a basis consisting of trapezohedral paths $\tau_m$ with $m \ge 2$ together with their merging images.
\end{theorem}

\begin{proof}

Using the definition of the graph $\Gamma$ and the edge generation 
procedure described above, we distinguish the following cases 
according to whether $\Gamma$ contains edges and whether it contains 
alternating cycles.

\medskip
\noindent\textbf{Case 1.}
If there exists an element $e_{a i j b} \in V$ such that $(a \to j \text{ or } a = j)$ and $(i \to b \text{ or } i = b)$, then by Lemma \ref{lem:image1}, $e_{a i j b} \in \Omega_3$. 
By Lemma \ref{lem:minimal}, there exists a scalar $c \in \mathbb{K}$ such that $\omega = c\, e_{a i j b}$. 

By Lemma \ref{lem:image1}, the path $\omega$ is a merging image of $c\, \tau_2$.

\medskip
\noindent\textbf{Case 2.}
Suppose there exists a cycle $v_0 v_1 \cdots v_{n-1} v_0$ in $\Gamma$ such that adjacent edges along the cycle have distinct colors. Without loss of generality, we may assume
$$v_i v_{i+1} \in E_1 \quad \text{for even } i, \quad v_i v_{i+1} \in E_2 \quad \text{for odd } i,$$
where indices are taken modulo $n$.

It is clear that $n$ is even, write $n = 2m$. 
If $m = 1$, then the cycle is $v_0 v_1 v_0$, and hence $v_0 v_1 \in E_1$ and $v_1 v_0 \in E_2$, which contradicts Claim \ref{clm:exclusion}. Hence, we must have $m \ge 2$. For $k = 0, 1, \dots, m-1$, Suppose that
$$v_{2k} = e_{a i_k j_k b}, \quad v_{2k+1} = e_{a i'_k j'_k b}.$$
Since $v_{2k} v_{2k+1} \in E_1$, we have $j'_k = j_k$. 
Since $v_{2k+1} v_{2k+2} \in E_2$, we have $i'_k = i_{k+1}$, where we identify $i_m = i_0$. 
Hence, $v_{2k+1} = e_{a i_{k+1} j_k b}.$ Let $\omega_0 = \sum_{k=0}^{n-1} (-1)^k \, v_k$, then by Lemma \ref{lem:image0}, $\omega_0 \in \Omega_3$. By Lemma \ref{lem:minimal}, there exists a scalar $c \in \mathbb{K}$ such that $\omega = c\, \omega_0$. Hence, $\omega$ can be written in the form
$$\omega = c\, (e_{a i_0 j_0 b} - e_{a i_1 j_0 b} + e_{a i_1 j_1 b} - \cdots + e_{a i_{m-1} j_{m-1} b} - e_{a i_0 j_{m-1} b}), $$
where the indices satisfy
$$a \to i_k \to j_k \to b,\quad i_{k+1} \to j_k.$$
By Lemma \ref{lem:image0}, the path $\omega$ is a merging image of $c\, \tau_{m}$.

\medskip
\noindent\textbf{Case 3.}
In this case, the graph $\Gamma$ contains no alternating cycle and no vertex in $V(\Gamma)$ satisfies the condition of Case 1, that is, there is no $e_{a i j b} \in V(\Gamma)$ such that both $(a \to j \text{ or } a = j)$ and $(i \to b \text{ or } i = b)$ hold.

Our strategy is to analyze the structure of $\Gamma$ by considering a maximal alternating path.
We will use Proposition \ref{clm:alternating} to show that such a path must terminate at vertices satisfying special boundary conditions, which will force $\omega$ to be a merging image of a trapezohedron. 

\begin{enumerate}
\item We first show that the graph $\Gamma$ contains at least one edge, that is, it does not consist entirely of isolated vertices.

Since Case 1 does not occur, for any vertex $e_{a i j b} \in V(\Gamma)$ we have either $(a \nrightarrow j$ and $a \ne j)$ or $(i \nrightarrow b$ and $i \ne b)$. Without loss of generality, assume $a \nrightarrow j$ and $a \ne j$. Then the term $e_{a j b}$ appearing in $\partial e_{a i j b}$ is non-allowed and must be canceled in $\partial \omega$ by the boundary of another elementary 3-path from $V$. Such a 3-path can only be of the form $e_{a i' j b}$ with $a \to i' \to j$. By definition, there is an edge connecting $e_{a i j b}$ and $e_{a i' j b}$. Hence, the graph $\Gamma$ contains edges.

\item Consider a maximal alternating path.

To understand the global structure of $\Gamma$, we consider a longest path $v_0 v_1 \cdots v_n$ such that the colors of adjacent edges along the path are different. The maximality of this path will impose strong restrictions on its endpoints. Without loss of generality, we may assume that
$$v_k v_{k+1} \in E_1 \text{ for even } k, \quad v_k v_{k+1} \in E_2 \text{ for odd } k.$$

Suppose that $v_0 = e_{a i j b}$. By Proposition \ref{clm:alternating}, we must have $i \to b$ or $i = b$. Similarly, let $v_n = e_{a i' j' b}$. If $n$ is odd, then $v_{n-1}v_n \in E_1$, and by Proposition \ref{clm:alternating} we obtain either $i' \to b$ or $i' = b$. If $n$ is even, then $v_{n-1}v_n \in E_2$, and Proposition \ref{clm:alternating} yields either $a \to j'$ or $a = j'$.

\medskip
\begin{itemize}
\item
If $n$ is odd, then $v_{n-1} v_n \in E_1$. Suppose that $n = 2m + 1$. 

For $k = 0, 1, \dots, m$, suppose that
$$v_{2k} = e_{a i_k j_k b}, \quad v_{2k+1} = e_{a i'_k j'_k b}.$$
From the discussion above, we have $(i_0 \to b$ or $i_0 = b)$ and $(i'_m \to b$ or $i'_m = b)$. 
Since $v_{2k} v_{2k+1} \in E_1$, we have $j'_k = j_k$. 
For $k \le m-1$, since $v_{2k+1} v_{2k+2} \in E_2$, we have $i'_k = i_{k+1}$. Rename $i'_m$ as $i_{m+1}$, hence, for $k = 0, 1, \dots, m$, $v_{2k+1} = e_{a i_{k+1} j_k b}$. (cf. Fig.\ref{fig:image2})

Let $\omega_0 = \sum_{k=0}^{n} (-1)^k \, v_k$, then by Lemma \ref{lem:image2}, $\omega_0 \in \Omega_3$. By Lemma \ref{lem:minimal}, there exists a scalar $c \in \mathbb{K}$ such that $\omega = c\, \omega_0$. 
Hence, $\omega$ can be written in the form
$$\omega = c\, (e_{a i_0 j_0 b} - e_{a i_1 j_0 b} + e_{a i_1 j_1 b} - \cdots + e_{a i_m j_m b} - e_{a i_{m+1} j_m b}), $$
where the indices satisfy
$$a \to i_k \to j_k \to b, \quad i_{k+1} \to j_k,\quad i_0 \to b \; or \; i_0 = b,\quad i_{m+1} \to b \; or \; i_{m+1} = b,\quad a \to i_{m+1},$$
by Lemma \ref{lem:image2}, the path $\omega$ is a merging image of $c\, \tau_{m+2}$.

\medskip
\item
If $n$ is even, then $v_{n-1} v_n \in E_2$. Suppose that $n = 2m$. 

For $k = 0, 1, \dots, m-1$, suppose that
$$v_{2k} = e_{a i_k j_k b}, \quad v_{2k+1} = e_{a i'_k j'_k b}, \quad v_{2m} = e_{a i_m j_m b}.$$
From the discussion above, we have $(i_0 \to b$ or $i_0 = b)$ and $(a \to j_m$ or $a = j_m)$. 
Since $v_{2k} v_{2k+1} \in E_1$, we have $j'_k = j_k$. 
Since $v_{2k+1} v_{2k+2} \in E_2$, we have $i'_k = i_{k+1}$. Hence,  $v_{2k+1} = e_{a i_{k+1} j_k b}$. (cf. Fig.\ref{fig:image3})

Let $\omega_0 = \sum_{k=0}^{n} (-1)^k \, v_k$, then by Lemma \ref{lem:image3}, $\omega_0 \in \Omega_3$. By Lemma \ref{lem:minimal}, there exists a scalar $c \in \mathbb{K}$ such that $\omega = c\, \omega_0$. 
Hence, $\omega$ can be written in the form
$$\omega = c\,(e_{a i_0 j_0 b} - e_{a i_1 j_0 b} + e_{a i_1 j_1 b} - \cdots - e_{a i_m j_{m-1} b} + e_{a i_m j_m b}),$$
where the indices satisfy
$$a \to i_k \to j_k \to b,\quad i_{k+1} \to j_k,\quad i_0 \to b \; or \; i_0 = b,\quad a \to j_m \; or \; a = j_m,$$
by Lemma \ref{lem:image3}, the path $\omega$ is a merging image of $c\, \tau_{m+2}$.
\end{itemize}
\end{enumerate}

\end{proof}

\section{The dimension and basis of $\Omega_3$}\label{sec:dimandbasis}

In this section, based on the results of the previous section, we deduce an explicit formula for determining the dimension of $\Omega_3(G)$, and moreover, we give an algorithm for explicitly finding a basis of $\Omega_3(G)$.

By Proposition \ref{pro:cluster}, $\Omega_3(G)$ admits a basis consisting of minimal $\partial$-invariant clusters. Hence it suffices to analyze the structure of minimal $(a,b)$-clusters.
For $a,b \in V(G)$, denote by $\Omega_3^{(a,b)}(G)$ the subspace of all
$\partial$-invariant $(a,b)$-clusters.
By Lemma \ref{lem:cluster}, we have the direct sum decomposition
$$\Omega_3(G) = \bigoplus_{a,b\in V(G)} \Omega_3^{(a,b)}(G).$$
Consequently,
$$\dim \Omega_3(G) = \sum_{a,b\in V(G)} \dim \Omega_3^{(a,b)}(G).$$

Throughout this section, we fix $a, b \in V(G)$. For convenience, we write $A := N^+(a) \setminus \{b\}$ and $B := N^-(b) \setminus \{a\}$ unless otherwise specified. For $\omega= \sum_{i_1, i_2} u^{i_1i_2} e_{i_0 i_1 i_2i_3}\in\Omega_3^{(a,b)}(G)$, define 
\begin{equation}\label{eq:Eomega}
  E_\omega:=\sum_{i_1\in N^+(a), i_2\in N^-(b)} u^{i_1i_2} e_{i_1 i_2}.  
\end{equation}
Recall that the induced digraph from $A \setminus B$ to $B \setminus A$ is defined in Definition \ref{def:ind}.

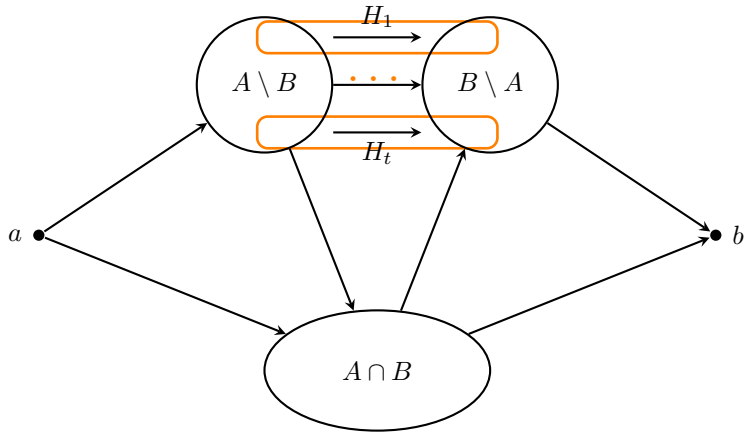
\begin{center}
\begin{tikzpicture}[>=stealth, thick]

\usetikzlibrary{fit,backgrounds}

% a 点
\node[circle, fill, inner sep=1.5pt, label=left:$a$] (a) at (0,0) {};

\node[draw, circle, minimum size=1.8cm] 
(AminusB) at (3,2) {$A\setminus B$};

\node[draw, circle, minimum size=1.8cm] 
(BminusA) at (6,2) {$B\setminus A$};

\node[draw, ellipse, minimum width=3cm, minimum height=1.6cm] 
(AinterB) at (4.5,-1.8) {$A\cap B$};

% b 点
\node[circle, fill, inner sep=1.5pt, label=right:$b$] (b) at (9,0) {};

% Arrows from a
\draw[->] (a) -- (AminusB);
\draw[->] (a) -- (AinterB);

% ===== 三条严格平行直线 =====

\coordinate (topL) at ([yshift=18pt]AminusB.east);
\coordinate (topR) at ([yshift=18pt]BminusA.west);

\coordinate (midL) at (AminusB.east);
\coordinate (midR) at (BminusA.west);

\coordinate (botL) at ([yshift=-18pt]AminusB.east);
\coordinate (botR) at ([yshift=-18pt]BminusA.west);

% 画箭头
\draw[->] (topL) -- (topR) node[midway, above] {$H_1$};

\draw[->] (midL) -- (midR)
      node[midway, yshift=2pt, inner sep=0pt, text=orange]
      {\fontsize{17pt}{17pt}\selectfont $\cdots$};

\draw[->] (botL) -- (botR) node[midway, below] {$H_t$};

% ===== 为两个框构造“延伸”坐标 =====

% 上框延伸（左右各多 8mm）
\coordinate (topBoxL) at ([xshift=-8mm]topL);
\coordinate (topBoxR) at ([xshift=8mm]topR);

% 下框延伸
\coordinate (botBoxL) at ([xshift=-8mm]botL);
\coordinate (botBoxR) at ([xshift=8mm]botR);

% ===== 空心框 =====
\begin{scope}[on background layer]

\node[
    draw=orange,
    line width=1pt,
    rounded corners=4pt,
    fit=(topBoxL)(topBoxR),
    inner sep=6pt
] {};

\node[
    draw=orange,
    line width=1pt,
    rounded corners=4pt,
    fit=(botBoxL)(botBoxR),
    inner sep=6pt
] {};

\end{scope}

% 向下箭头
\draw[->] (AminusB) -- (AinterB);
\draw[->] (AinterB) -- (BminusA);

% 指向 b
\draw[->] (BminusA) -- (b);
\draw[->] (AinterB) -- (b);

\end{tikzpicture}
\end{center}

\begin{definition}\label{def:edge}
Let $G=(V,E)$ be a directed graph. For any subsets $A,B \subseteq V$, we define
$$E(A,B) = \{\, (u,v)\in E \mid u\in A,\ v\in B \,\}.$$
\end{definition}

 A formula for computing $\dim \Omega_3^{(a,b)}(G)$ is given as follows:

\begin{theorem}\label{thm:main2}
Let $H$ be the induced digraph from $A\setminus B$ to $B\setminus A$, and let $t=t(a,b)$ be the number of its connected components. Write these components as $H_1, \dots, H_t$. For each $k\in\{1,...,t\}$. Define $S_k = S_k(a,b)$ to be the set of edges that either start in $V(H_k) \setminus B$ and end in $ (N^{+}(a) \cup \{ a \} \bigr) \cap N^{-}(b)$, or start in $N^{+}(a) \cap \bigl( N^{-}(b) \cup \{ b \} \bigr)$ and end in $V(H_k) \setminus A$. Then
$$\begin{aligned}
\dim \Omega_3^{(a,b)}(G) =& (|E(H)|-|V(H)|+t) + |E\bigl( N^{+}(a) \cap \bigl( N^{-}(b) \cup \{ b \} \bigr), \bigl( N^{+}(a) \cup \{ a \} \bigr) \cap N^{-}(b) \bigr)|\\
&+ \sum_{k=1}^t \max\{0, |S_k|-1\}.
\end{aligned}$$
\end{theorem}

As a direct corollary, by summing over all $a, b \in V(G)$, we obtain an explicit formula for the dimension of $\Omega_3(G)$, which can be stated as follows:

\begin{theorem}\label{thm:main3}
The dimension of $\Omega_3(G)$ is the sum of $\operatorname{dim}\Omega_3^{a,b}(G)$ for all pairs of vertices $(a,b)$. More precisely,
\begin{equation}
\begin{aligned}
    \dim \Omega_3(G) =& \sum_{a,b\in V(G)} \left(|E(\operatorname{Ind}(N^+(a)\setminus N^-(b), N^-(b)\setminus N^+(a)))|\right.\\
    &-\left.|V(\operatorname{Ind}(N^+(a)\setminus N^-(b), N^-(b)\setminus N^+(a)))|+t(a,b)\right)\\
    &+ \sum_{a,b\in V(G)}|E\bigl( N^{+}(a) \cap \bigl( N^{-}(b) \cup \{ b \} \bigr), \bigl( N^{+}(a) \cup \{ a \} \bigr) \cap N^{-}(b) \bigr)|\\
    &+ \sum_{a,b\in V(G)}\sum_{k=1}^{t(a,b)} \max\{0, |S_k(a,b)|-1\}. 
\end{aligned}
\end{equation}
\end{theorem}

%在陈述这个定理之前，首先应该将这个定理背后的一些定义写出来，陈述这个定理需要一个subsection。然后把定理讲清楚之后，证明可以放到下一个subsection。

%我还没改完，先改到这里

The structure of this section is organized around the proof of Theorem \ref{thm:main2}. We begin by introducing the notion of \textit{terminal elements}. An elementary term $e_{aijb}$ is called a \textit{terminal element} if it satisfies at least one of the following two conditions: (i) $a \rightarrow j$ or $a = j$; (ii) $i \rightarrow b$ or $i = b$. More precisely, we define the \textit{terminal number} of $e_{aijb}$ to be the number of conditions among (i) and (ii) that hold. Thus each $e_{aijb}$ has terminal number $0$, $1$, or $2$. %According to the analysis in Section \ref{subsec2:structureomega3}, each generator contains at most two terminal elements with nonzero coefficients. 

Note that every elementary term $e_{aijb}\in A_3(G)$ satisfies $i\in N^+(a)$ and $j\in N^-(b)$, hence $(a,b)$-clusters correspond to subsets of directed edges from $A$ to $B$. According to the structural classification in Section \ref{subsec2:structureomega3}, we divide minimal generators into three types due to the number of terminal elements:
\begin{enumerate}
\item No terminal elements, which is discussed in Section \ref{sec:basis1}, it corresponds to a trapezohedron with endpoints $a$ and $b$, whose remaining vertices lie in $A \setminus B$ and $B \setminus A$.
\item Exactly one terminal element with terminal number 2, which is discussed in Section \ref{sec:basis2}. More precisely, each such generator consists of a single term $e_{aijb}$ with terminal number $2$, where $(i,j)$ is an edge inside $(A \cap B) \cup \{a,b\}$.
\item Exactly two terminal elements with terminal number 1, which is discussed in Section \ref{sec:basis3}. More precisely, each such generator contains two edges from $A \setminus B$ or $B \setminus A$ to $(A \cap B) \cup \{a,b\}$, while all other edges from $A$ to $B$ lie inside $\operatorname{Ind}(A \setminus B, B \setminus A)$.
\end{enumerate}
Note that in each case, we need to consider two subcases: $a \nrightarrow b$ and $a \rightarrow b$, which require slightly different treatments. Moreover, the case $a = b$ cannot be omitted; it is handled in the same way as the case $a \nrightarrow b$.

\subsection{Generators without terminal elements}\label{sec:basis1}

Concerning the number of terminal elements in a generator, note that the discussion in this case remains valid regardless of whether $a \to b$, $a \nrightarrow b$, or $a = b$. If a generator has no terminal elements, then every term $e_{a i j b}$ with nonzero coefficient must satisfy $i \nrightarrow b$ and $i \neq b$, as well as $a \nrightarrow j$ and $a \neq j$. Consequently, such terms arise exclusively from the directed bipartite graph from $A \setminus B$ to $B \setminus A$. Let $H$ be the induced subgraph $\operatorname{Ind}(A \setminus B, B \setminus A)$, and we regard $H$ as an undirected graph when discussing cycles.

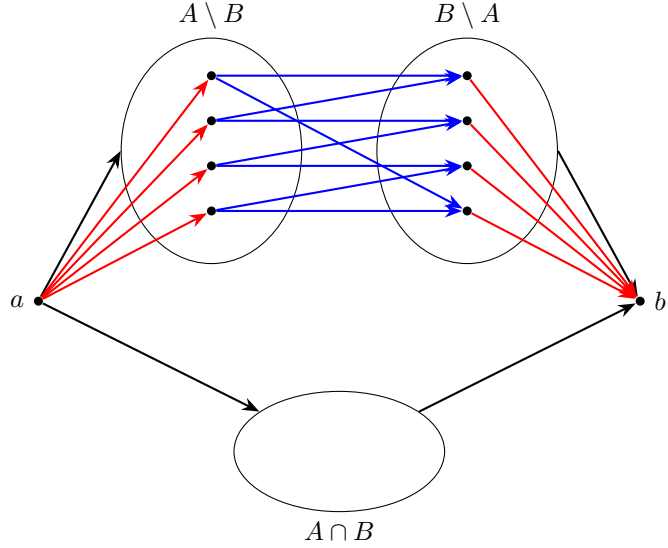
\begin{center}
\begin{tikzpicture}[
    dot/.style={circle,fill,inner sep=1.2pt},
    set/.style={draw,ellipse,minimum width=2.4cm,minimum height=3cm},
    edge/.style={-Stealth,thick},
    rededge/.style={-Stealth,red,thick},
    blueedge/.style={-Stealth,blue,thick},
    node distance=1cm
]

% a,b
\node[dot,label=left:$a$] (a) at (-4,0) {};
\node[dot,label=right:$b$] (b) at (4,0) {};

% sets (A\B,B\A) 分开一点
\node[set,label=above:$A\setminus B$] (A) at (-1.7,2) {};
\node[set,label=above:$B\setminus A$] (B) at (1.7,2) {};
\node[set,minimum width=2.8cm,minimum height=1.6cm,
      label=below:$A\cap B$] (C) at (0,-2) {};

% points in A\B 竖着排列
\node[dot] (A1) at (-1.7,3.0) {};
\node[dot] (A2) at (-1.7,2.4) {};
\node[dot] (A3) at (-1.7,1.8) {};
\node[dot] (A4) at (-1.7,1.2) {};

% points in B\A 竖着排列
\node[dot] (B1) at (1.7,3.0) {};
\node[dot] (B2) at (1.7,2.4) {};
\node[dot] (B3) at (1.7,1.8) {};
\node[dot] (B4) at (1.7,1.2) {};

% 黑色结构箭头
\draw[edge] (a) -- ([xshift=-0pt,yshift=-0pt]A.west);
\draw[edge] ([xshift=-0pt,yshift=-0pt]B.east) -- (b); 
\draw[edge] (a) -- (C); % a -> A∩B
\draw[edge] (C) -- (b); % A∩B -> b

% 红色箭头
\draw[rededge] (a) -- (A1);
\draw[rededge] (a) -- (A2);
\draw[rededge] (a) -- (A3);
\draw[rededge] (a) -- (A4);

\draw[rededge] (B1) -- (b);
\draw[rededge] (B2) -- (b);
\draw[rededge] (B3) -- (b);
\draw[rededge] (B4) -- (b);

% 蓝色箭头 A -> B
\draw[blueedge] (A1) -- (B1);
\draw[blueedge] (A2) -- (B2);
\draw[blueedge] (A3) -- (B3);
\draw[blueedge] (A4) -- (B4);
\draw[blueedge] (A2) -- (B1);
\draw[blueedge] (A3) -- (B2);
\draw[blueedge] (A4) -- (B3);
\draw[blueedge] (A1) -- (B4);

\end{tikzpicture}
\end{center}

\begin{proposition}\label{pro:0}
Each cycle $C$ in $H$ gives rise to an $(a,b)$-cluster $\omega_C$ with no terminal elements. Conversely, every minimal generator without terminal elements corresponds to a cycle in $H$.
\end{proposition}

\begin{proof}
\begin{comment}
    Let $C = i_0 j_0 \cdots i_{m-1} j_{m-1} i_0$
be a cycle in $H$ such that $i_k \in A \setminus B$ and $j_k \in B \setminus A$. 
Then, for every $k = 0,1,\ldots,m-1$, we have
$$i_k \to j_k \quad \text{and} \quad i_{k+1} \to j_k,$$
where the indices are taken modulo $m$. 
By Lemma \ref{lem:image0}, the alternating sum
$$\omega_C= e_{a i_0 j_0 b} - e_{a i_1 j_0 b} + e_{a i_1 j_1 b} - \cdots + e_{a i_{m-1} j_{m-1} b} - e_{a i_0 j_{m-1} b}$$
belongs to $\Omega_3(G)$ and contains no terminal elements.
\end{comment}

Let $\omega$ be a minimal generator without terminal elements. By the structural description given in Section \ref{subsec2:structureomega3}, every vertex in the corresponding graph\footnote{The corresponding graph means the edges with nonzero coefficient in $E_\omega$.} has degree at least two. Hence, the graph contains a cycle $C$. By the preceding argument, $\omega_C \in \Omega_3(G)$. It then follows from Lemma \ref{lem:minimal} that there exists a scalar $c \in K$ such that $\omega = c\, \omega_C$. Therefore, $\omega$ corresponds to a cycle.
\end{proof}

Let $F$ be a spanning forest of $H$. For each edge $e \in E(H) \setminus E(F)$, there exists a unique cycle $C_e$ in $H$ such that
$$e \in C_e \quad \text{and} \quad C_e \setminus \{e\} \subseteq E(F).$$
This cycle $C_e$ is called the fundamental cycle determined by $e$ 
with respect to the spanning forest $F$. Since linear combinations of $(a,b)$-clusters without terminal 
elements again have no terminal elements, the collection of all 
such clusters forms a vector space, which we denote by 
$\Omega_3^{0}$.\footnote{Although $\Omega_3^{0}$ depends on $a$ and $b$, we omit them from the notation for convenience.} The following lemma provides a basis of $\Omega_3^{0}$.

\begin{lemma}\label{lem:b0}
$\mathcal{B}_0=\{\omega_{C_e} \mid e \in E(H) \setminus E(F)\}$ is a basis of $\Omega_3^{0}$.
\end{lemma}

\begin{proof}
For each edge $e = (i,j) \in E(H) \setminus E(F)$, let $C_e$ denote the fundamental cycle determined by $e$ with respect to the spanning forest $F$, and assume that the coefficient of $e_{aijb}$ in the representation of $\omega_{C_e}$ equals $+1$. To prove linear independence, suppose that
$$\sum_{e \in E(H)\setminus E(F)} \lambda_e \, \omega_{C_e} = 0.$$
Fix $e_0 \in E(H)\setminus E(F)$. 
By construction, the edge $e_0$ belongs to $C_{e_0}$ and does not 
belong to any other cycle $C_e$ with $e \neq e_0$. 
It follows that the coefficient of $e_0$ in the above sum equals 
$\lambda_{e_0}$, hence $\lambda_{e_0}=0$. 
Since $e_0$ was arbitrary, all coefficients vanish, and thus 
$\mathcal{B}_0$ is linearly independent. For the spanning property, by Lemma \ref{lem:cluster}, it suffices 
to show that every minimal $(a,b)$-cluster without terminal elements 
can be expressed as a linear combination of the elements in 
$\mathcal{B}_0$.

Let $\omega= \sum_{i_1\in A, i_2\in B} u^{i_1i_2} e_{a i_1 i_2b}\in\Omega_3^{(a,b)}(G)$ be a minimal $(a,b)$-cluster. Recall that $ E_\omega:=\sum_{i_1\in A, i_2\in B} u^{i_1i_2} e_{i_1 i_2}$ is defined in \eqref{eq:Eomega} and represents a linear combination of edges in $\operatorname{Ind}(A \setminus B, B \setminus A)$. Define $$\omega':=\omega-\sum_{i_1\in A, i_2\in B,(i_1,i_2)\notin E(F)}u^{i_1i_2}\omega_{C_{(i_1,i_2)}}\in\Omega_3^{(a,b)}(G).$$ 
 
It remains to prove that $\omega' = 0$. We first show that if $(i_1,i_2) \in E(H) \setminus E(F)$, then the coefficient of $e_{a i_1 i_2 b}$ in $\omega'$ vanishes. Indeed, for each edge $(i_1,i_2) \in E(H) \setminus E(F)$, the coefficient of $e_{a i_1 i_2 b}$ in $\omega$ is $u^{i_1 i_2}$, while its coefficient in $\omega_{C_{(i_1,i_2)}}$ is $1$. Moreover, for any distinct edge $(i_1',i_2') \neq (i_1,i_2)$ in $E(H) \setminus E(F)$, the coefficient of $e_{a i_1 i_2 b}$ in $\omega_{C_{(i_1',i_2')}}$ is $0$.

Second, we show that for each $i_1 \in A \setminus B$, the sum of the coefficients of $e_{a i_1 i_2 b}$ in $\omega'$ over all $i_2 \in B \setminus A$ with $(i_1,i_2) \in E(H)$ is zero. Indeed, this property holds for $\omega$ and for every $\omega_{C_{(i_1,i_2)}}$ because $E_\omega$ and each $C_{(i_1,i_2)}$ are cycles. Similarly, for each $i_2 \in B \setminus A$, the sum of the coefficients of $e_{a i_1 i_2 b}$ in $\omega'$ over all $i_1 \in A \setminus B$ with $(i_1,i_2) \in E(H)$ is zero.

Together with the two facts established above, we conclude that $\omega' = 0$. Indeed, suppose $\omega' \neq 0$. Then the set of edges with nonzero coefficients in $\omega'$ is contained in $E(F)$, and therefore forms a subgraph of the forest $F$. Consequently, there exists a vertex $v \in V(F)$ that is incident to exactly one edge with a nonzero coefficient in $\omega'$. This, however, contradicts the second fact. Therefore, $\mathcal{B}_0$ spans $\Omega_3^{0}$.
\end{proof}

Since the number of edges in a spanning forest of a graph equals the number of vertices minus the number of connected components, we obtain the following corollary:
\begin{corollary}\label{cor:b0}
Let $E(H)$, $V(H)$ denote the edge and vertex sets of $H$,
and let $t$ be the number of connected components of $H$.
Then 
$$|\mathcal{B}_0| = |E(H)| - |V(H)| + t.$$
\end{corollary}

\subsection{Generators with exactly one terminal element}\label{sec:basis2}

This subsection studies the generators of $\Omega_3$ with exactly one terminal element. 

According to the analysis in Section \ref{subsec2:structureomega3}, if there is only one terminal element $e_{a i j b}$ in a generator, then we must have $a \to i \to j \to b$ and $(a \to j \text{ or } a = j)$ and $(i \to b \text{ or } i = b)$. Equivalently,
$$(i,j) \in E\bigl( N^{+}(a) \cap \bigl( N^{-}(b) \cup \{ b \} \bigr), \bigl( N^{+}(a) \cup \{ a \} \bigr) \cap N^{-}(b) \bigr).$$

If $a \nrightarrow b$, then $a \notin N^{-}(b)$ and $b \notin N^{+}(a)$. Consequently, the term $e_{a i j b}$ satisfies $i, j \in A \cap B$. If $a \rightarrow b$, then $a \in N^{-}(b)$ and $b \in N^{+}(a)$. Consequently, the term $e_{a i j b}$ satisfies $i \in (A \cap B) \cup \{b\}$ and $j \in (A \cap B) \cup \{a\}$.

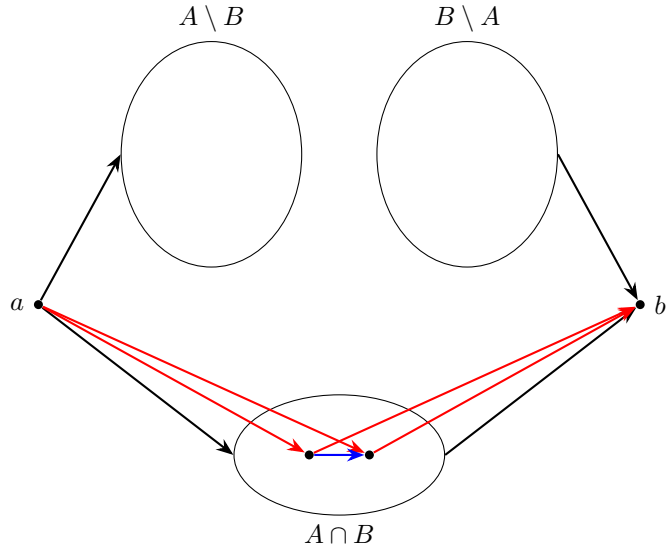
\begin{center}
\begin{tikzpicture}[
    dot/.style={circle,fill,inner sep=1.2pt},
    set/.style={draw,ellipse,minimum width=2.4cm,minimum height=3cm},
    edge/.style={-Stealth,thick},
    rededge/.style={-Stealth,red,thick},
    blueedge/.style={-Stealth,blue,thick},
    node distance=1cm
]

% a,b
\node[dot,label=left:$a$] (a) at (-4,0) {};
\node[dot,label=right:$b$] (b) at (4,0) {};

% sets (A\B,B\A) 分开一点
\node[set,label=above:$A\setminus B$] (A) at (-1.7,2) {};
\node[set,label=above:$B\setminus A$] (B) at (1.7,2) {};
\node[set,minimum width=2.8cm,minimum height=1.6cm,
      label=below:$A\cap B$] (C) at (0,-2) {};

\node[dot] (C1) at (-0.4,-2) {};
\node[dot] (C2) at (0.4,-2) {};

% 黑色结构箭头
\draw[edge] (a) -- ([xshift=-0pt,yshift=-0pt]A.west);
\draw[edge] ([xshift=-0pt,yshift=-0pt]B.east) -- (b); 
\draw[edge] (a) -- (C.west); % a -> A∩B
\draw[edge] (C.east) -- (b); % A∩B -> b

\draw[rededge] (a) -- (C1);
\draw[rededge] (C2) -- (b);
\draw[rededge] (a) -- (C2);
\draw[rededge] (C1) -- (b);

\draw[blueedge] (C1) -- (C2);

\end{tikzpicture}
\end{center}

\begin{proposition}\label{pro:1}
Each directed edge 
$$(i,j) \in E\bigl( N^{+}(a) \cap \bigl( N^{-}(b) \cup \{ b \} \bigr), \bigl( N^{+}(a) \cup \{ a \} \bigr) \cap N^{-}(b) \bigr)$$
determines a minimal generator of $\Omega_3^{(a,b)}(G)$, and these generators are linearly independent.  
\end{proposition}

\begin{proof}
If $(i,j) \in E\bigl( N^{+}(a) \cap \bigl( N^{-}(b) \cup \{ b \} \bigr), \bigl( N^{+}(a) \cup \{ a \} \bigr) \cap N^{-}(b) \bigr)$, then $a \to i \to j \to b$ and  $(a \to j \text{ or } a = j)$ and $(i \to b \text{ or } i = b)$. By Lemma \ref{lem:image0}, we have $e_{a i j b} \in \Omega_3(G)$. Distinct edges $i \to j$ give rise to distinct elementary basis elements $e_{a i j b}$ in $A_3(G)$. Consequently, the corresponding generators are linearly independent.
\end{proof}

Let 
$$\mathcal{B}_1 = \{ e_{a i j b} \mid (i,j) \in E\bigl( N^{+}(a) \cap \bigl( N^{-}(b) \cup \{ b \} \bigr), \bigl( N^{+}(a) \cup \{ a \} \bigr) \cap N^{-}(b) \bigr) \}.$$ 
We have $|\mathcal{B}_1|=|E\bigl( N^{+}(a) \cap \bigl( N^{-}(b) \cup \{ b \} \bigr), \bigl( N^{+}(a) \cup \{ a \} \bigr) \cap N^{-}(b) \bigr)|$. Since the elements in $\mathcal{B}_0$ do not involve any terminal elements, it follows that $\mathcal{B}_0 \cup \mathcal{B}_1$ is linearly independent.

\subsection{Generators with two terminal elements}\label{sec:basis3}
This subsection studies generators containing two terminal elements. Such a generator $\omega$ can be described as follows. Recall that $E_\omega$ is defined in \eqref{eq:Eomega}. The edges in $G$ with nonzero coefficients in $E_\omega$ form a path\footnote{This path is different from the path of digraphs defined above, it do not require the direction of each edge.} $ i_0 i_1, i_1 i_2, \dots, i_{t-1} i_t $, where $i_0\in( N^{+}(a) \cup \{ a \} \bigr) \cap N^{-}(b)$, $i_t\in\bigl(  N^{+}(a) \cap \bigl( N^{-}(b) \cup \{ b \} \bigr)$ and $i_1, \dots, i_{t-1} \in (A \setminus B) \cup (B \setminus A)$, note that the direction of edges are ignored in this setting. The coefficients of these edges are $\pm 1$, and the coefficients of any two adjacent edges have opposite signs.

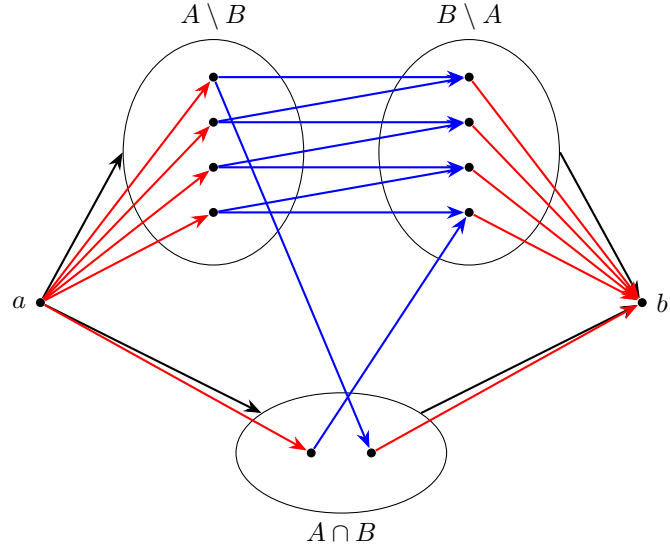
\begin{center}
\begin{tikzpicture}[
    dot/.style={circle,fill,inner sep=1.2pt},
    set/.style={draw,ellipse,minimum width=2.4cm,minimum height=3cm},
    edge/.style={-Stealth,thick},
    rededge/.style={-Stealth,red,thick},
    blueedge/.style={-Stealth,blue,thick},
    node distance=1cm
]

% a,b
\node[dot,label=left:$a$] (a) at (-4,0) {};
\node[dot,label=right:$b$] (b) at (4,0) {};

% sets (A\B,B\A) 分开一点
\node[set,label=above:$A\setminus B$] (A) at (-1.7,2) {};
\node[set,label=above:$B\setminus A$] (B) at (1.7,2) {};
\node[set,minimum width=2.8cm,minimum height=1.6cm,
      label=below:$A\cap B$] (C) at (0,-2) {};

% points in A\B 竖着排列
\node[dot] (A1) at (-1.7,3.0) {};
\node[dot] (A2) at (-1.7,2.4) {};
\node[dot] (A3) at (-1.7,1.8) {};
\node[dot] (A4) at (-1.7,1.2) {};

% points in B\A 竖着排列
\node[dot] (B1) at (1.7,3.0) {};
\node[dot] (B2) at (1.7,2.4) {};
\node[dot] (B3) at (1.7,1.8) {};
\node[dot] (B4) at (1.7,1.2) {};

\node[dot] (C1) at (-0.4,-2) {};
\node[dot] (C2) at (0.4,-2) {};

% 黑色结构箭头
\draw[edge] (a) -- ([xshift=-0pt,yshift=-0pt]A.west);
\draw[edge] ([xshift=-0pt,yshift=-0pt]B.east) -- (b); 
\draw[edge] (a) -- (C); % a -> A∩B
\draw[edge] (C) -- (b); % A∩B -> b

% 红色箭头
\draw[rededge] (a) -- (A1);
\draw[rededge] (a) -- (A2);
\draw[rededge] (a) -- (A3);
\draw[rededge] (a) -- (A4);

\draw[rededge] (B1) -- (b);
\draw[rededge] (B2) -- (b);
\draw[rededge] (B3) -- (b);
\draw[rededge] (B4) -- (b);

\draw[rededge] (a) -- (C1);
\draw[rededge] (C2) -- (b);

% 蓝色箭头 A -> B
\draw[blueedge] (A1) -- (B1);
\draw[blueedge] (A2) -- (B2);
\draw[blueedge] (A3) -- (B3);
\draw[blueedge] (A4) -- (B4);
\draw[blueedge] (A2) -- (B1);
\draw[blueedge] (A3) -- (B2);
\draw[blueedge] (A4) -- (B3);
\draw[blueedge] (A1) -- (C2);
\draw[blueedge] (C1) -- (B4);

\end{tikzpicture}
\end{center}

Suppose that $H=\operatorname{Ind}(A\setminus B,B\setminus A)$ is decomposed into connected components
$$H = H_1 \sqcup \cdots \sqcup H_t.$$
For each $k = 1,\ldots,t$, we have
$$H_k = \operatorname{Ind}\big[(A \setminus B)_k , (B \setminus A)_k\big],$$
where $(A \setminus B)_k$ and $(B \setminus A)_k$ represents the intersection of $V(H_k)$ and $A \setminus B$ and $B \setminus A$. For each component $H_k$, define
$$\Delta_k^1 = E \bigl( (A\setminus B)_k,\ \bigl( N^{+}(a) \cup \{ a \} \bigr) \cap N^{-}(b) \bigr),$$
$$\Delta_k^2 = E \bigl(  N^{+}(a) \cap \bigl( N^{-}(b) \cup \{ b \} \bigr),\  (B\setminus A)_k \bigr).$$

Note that the edge set $E(A,B)$ is defined in Definition \ref{def:edge}. Equivalently, if $a \nrightarrow b$,
$$\Delta_k^1 = E \bigl(  (A\setminus B)_k,\  A\cap B \bigr),$$
$$\Delta_k^2 = E \bigl(  A\cap B,\  (B\setminus A)_k \bigr),$$
if $a \rightarrow b$,
$$\Delta_k^1 = E \bigl(  (A\setminus B)_k,\  (A\cap B) \cup \{ a \}  \bigr),$$
$$\Delta_k^2 = E \bigl(  (A\cap B) \cup \{ b \},\  (B\setminus A)_k \bigr).$$

Set $S_k=\Delta_k^1\cup\Delta_k^2$. We have the following proposition, which states that every pair of distinct edges determines a generator:
\begin{proposition}
If $|S_k|=n_k\ge2$, every unordered pair of distinct edges in $S_k$ determines a generator of $\Omega_3^{(a,b)}(G)$.
\end{proposition}

\begin{proof}
If $|\Delta_k^1 \cup \Delta_k^2|\ge2$, let $(i_0,j_0)$ and $(i',j')$ be an unordered pair of distinct edges in $\Delta_k^1 \cup \Delta_k^2$. If both $(i_0,j_0)$ and $(i',j')$ lie in $\Delta_k^2$, then $i_0,i' \in N^{+}(a) \cap \bigl( N^{-}(b) \cup \{ b \} \bigr)$ and $j_0,j' \in (B \setminus A)_k$. By connectivity, there exists a path $j_0 i_1 j_1 \cdots i_m j'$,\footnote{This path is different from the path of digraphs defined above, it do not require the direction of each edge.} where $i_t \in (A \setminus B)_k$ and $j_t \in (B \setminus A)_k$ for all relevant indices $t$. Writing $(i',j') = (i_{m+1}, j_m)$, we have $i_t \to j_t \text{ and } i_{t+1}\to j_t  \text{ for all } t=0,1,\ldots,m.$
Hence, by Lemma \ref{lem:image3},
$$\omega = e_{a i_0 j_0 b} - e_{a i_1 j_0 b} + e_{a i_1 j_1 b} - \cdots + e_{a i_m j_m b} - e_{a i_{m+1} j_m b} \in \Omega_3.$$
The case $(i_0,j_0)$ and $(i',j')$ lie in $\Delta_k^1$ can be treated analogously.

If $(i_0,j_0)$ lie in $\Delta_k^2$ and $(i',j')$ lie in $\Delta_k^1$, then $i_0\in N^{+}(a) \cap \bigl( N^{-}(b) \cup \{ b \} \bigr), j' \in \bigl( N^{+}(a) \cup \{ a \} \bigr) \cap N^{-}(b) $ and $j_0 \in (B \setminus A)_k$, $i' \in (A \setminus B)_k$.  By connectivity, there exists a path $j_0 i_1 j_1 \cdots j_{m-1} i'$, where $i_t \in (A \setminus B)_k$ and $j_t \in (B \setminus A)_k$ for all relevant indices $t$. Writing $(i',j') = (i_m, j_m)$, we have
$i_t \to j_t  \text{ for all } t=0,1,\ldots,m  \text{ and }  i_{t+1} \to j_t \text{ for all } t=0,1,\ldots,m-1.$
Hence, by Lemma \ref{lem:image2},
$$\omega = e_{a i_0 j_0 b} - e_{a i_1 j_0 b} + e_{a i_1 j_1 b} - \cdots + e_{a i_m j_m b} - e_{a i_{m+1} j_m b} \in \Omega_3.$$

For different paths, the coefficients of $e_{a i_0 j_0 b}$ and 
$e_{a i' j' b}$ in the corresponding $\omega$ remain unchanged. 
Consequently, their difference contains no terminal elements and is 
therefore linearly dependent on the elements of $\mathcal{B}_0$.
Hence, every unordered pair of distinct edges in $\Delta_k^1 \cup \Delta_k^2$ determines a generator of $\Omega_3(G)$.
\end{proof}

Enumerate the edges in $\Delta_k^1 \cup \Delta_k^2$ as $(i^k_0,j^k_0),\ldots,(i^k_{n_k-1},j^k_{n_k-1}).$
For each $s \in \{1,2,\ldots,n_k-1\}$, let $\omega^k_s$ denote the element of $\Omega^{(a,b)}_3(G)$ determined by the pair $(i^k_0,j^k_0)$ and $(i^k_s,j^k_s)$. 
Then, for any $s,t \in \{1,\ldots,n_k-1\}$, the element determined by 
$(i^k_s,j^k_s)$ and $(i^k_t,j^k_t)$ can be expressed as $\omega^k_s - \omega^k_t$. Let 
$$\mathcal{B}_2 = \bigcup_{k=1}^t \{ \omega_s^k \mid 1 \le s \le n_k - 1 \}.$$
We compute the number of elements in $\mathcal{B}_2$.

\begin{corollary}\label{cor:b2}
$$|\mathcal{B}_2|=\sum_{k=1}^t \max\{0, |S_k|-1\}.$$
\end{corollary}

\subsection{Completion of the proof}\label{sec:basis4}

We complete the proof of Theorem \ref{thm:main2} in this subsection. It remains to show that the elements of $\mathcal{B}_0 \cup \mathcal{B}_1 \cup \mathcal{B}_2$ form a basis of $\Omega_3(G)$. In other words, we must prove that these elements are linearly independent in Lemma \ref{lem:Bindependent} and that they span $\Omega_3(G)$ in Lemma \ref{lem:Blinearspan}.

\begin{lemma}\label{lem:Bindependent}
The set 
$$\mathcal{B}_{a b}=\mathcal{B}_0 \cup \mathcal{B}_1 \cup \mathcal{B}_2$$
is linearly independent.
\end{lemma}

\begin{proof}

We have already shown in Section \ref{sec:basis2} that $\mathcal{B}_0 \cup \mathcal{B}_1$ is linearly independent. It remains to prove that this independence is preserved after adding the elements of $\mathcal{B}_2$. For each pair $(s,k)$ with $1 \le s \le n_k - 1$, the term $e_{a i_s^k j_s^k b}$ appears only in $\omega_s^k$ among all elements of $\mathcal{B}_0 \cup \mathcal{B}_1 \cup \mathcal{B}_2$.  Therefore, in any linear relation among the elements of $\mathcal{B}_0 \cup \mathcal{B}_1 \cup \mathcal{B}_2$, the coefficient of $\omega_s^k$ must vanish. It follows that all coefficients are zero, and hence the union $B_{a b}$ is linearly independent.
\end{proof}

\begin{lemma}\label{lem:Blinearspan}
Every minimal $\partial$-invariant $(a,b)$-cluster belongs to the 
linear span $\langle \mathcal{B}_{ab} \rangle$.
\end{lemma}

\begin{proof}
    According to the analysis above, every minimal $\partial$-invariant cluster with no terminal elements belongs to $\mathcal{B}_0$, and every such cluster with exactly one terminal element belongs to $\mathcal{B}_1$.

    The remaining case is the minimal $\partial$-invariant cluster $\omega$ with two terminal elements. In this case, $E_\omega$ forms an undirected path whose endpoints lie in $A \cap B$ and whose internal vertices lie in $A \setminus B$ or $B \setminus A$. Consequently, both the first and last edges of the path go from the same connected component of $\operatorname{Ind}(  A \setminus B,B \setminus A)$ to $A \cap B$. In this case, every minimal $\partial$-invariant cluster with two terminal elements is linearly dependent modulo $\mathcal{B}_0$ on an element of the form $\omega_s^k - \omega_t^k$. In particular, every minimal $\partial$-invariant $(a,b)$-cluster with two terminal elements lies in the linear span $\langle \mathcal{B}_0 \cup \mathcal{B}_2 \rangle$.
\end{proof}

\subsection{An algorithm for computing basis of $\Omega_3$}\label{sec:algorithm}

In this subsection, we present an explicit algorithm for computing a basis of $\Omega_3(G)$.

\medskip
\noindent\textbf{Input:} A directed graph $G = (V,E)$.

\noindent\textbf{Output:} A basis of the vector space $\Omega_3(G)$.

\medskip
\noindent\textbf{Algorithm.}

\begin{enumerate}

\item[(I)] \textbf{Local computation for fixed $(a,b)$.}

For each ordered pair $(a,b) \in V(G)\times V(G)$, compute a basis $B_{a,b}$ of $\Omega_3^{(a,b)}(G)$ as follows:

\begin{enumerate}

\item[(i)] Define
$A := \{ i \in V(G) \mid a \to i \}\setminus\{b\},  B := \{ j \in V(G) \mid j \to b \}\setminus\{a\}.$

Decompose
$A \cap B,  A \setminus B, B \setminus A.$

Construct the bipartite graph
$H : A \setminus B \longrightarrow B \setminus A,$
whose edges are all arrows
$i \to j  \text{ with }  i \in A \setminus B, \; j \in B \setminus A.$
When counting connected components and cycles,regard $H$ as an undirected graph.

\item[(ii)] \textbf{Generators without terminal elements.}

Compute the connected components $H = H_1 \sqcup \cdots \sqcup H_t$ of $H$.

Then compute a spanning forest $F$ of $H$.

For each edge $e \in E(H) \setminus E(F)$, determine the fundamental cycle 
$C_e=i_0 j_0 \cdots i_{m-1} j_{m-1} i_0  (i_k \in A \setminus B, j_k \in B \setminus A,)$
and construct the alternating sum 
$\omega_{C_e} = e_{a i_0 j_0 b} - e_{a i_1 j_0 b} + e_{a i_1 j_1 b} - \cdots + e_{a i_{m-1} j_{m-1} b} - e_{a i_0 j_{m-1} b}$
as in Section \ref{sec:basis1}.

Let 
$\mathcal{B}_0 := \{ \omega_{C_e} : e\in E(H) \setminus E(F) \}.$

\item[(iii)] \textbf{Generators with one terminal element.}

For each directed edge
$(i , j) \in E\bigl( N^{+}(a) \cap \bigl( N^{-}(b) \cup \{ b \} \bigr), \bigl( N^{+}(a) \cup \{ a \} \bigr) \cap N^{-}(b) \bigr),$
add the generator $e_{aijb}$.

Let 
$\mathcal{B}_1 := \{ e_{aijb} \mid (i , j) \in E\bigl( N^{+}(a) \cap \bigl( N^{-}(b) \cup \{ b \} \bigr), \bigl( N^{+}(a) \cup \{ a \} \bigr) \cap N^{-}(b) \bigr) \}.$

\item[(iv)] \textbf{Generators with two terminal elements.}

Let
$H = H_1 \sqcup \cdots \sqcup H_t$
be the connected components of $H$, where
$H_k = H[(A \setminus B)_k \cup (B \setminus A)_k], \quad (A \setminus B)_k \ \subseteq A \setminus B, \quad (B \setminus A)_k \ \subseteq B \setminus A$.

For each component $H_k$, define
$S_k = \Delta_k^1 \cup \Delta_k^2,$
where
$\Delta_k^1 = E \bigl( (A\setminus B)_k,\ \bigl( N^{+}(a) \cup \{ a \} \bigr) \cap N^{-}(b) \bigr),$
$\Delta_k^2 = E \bigl(  N^{+}(a) \cap \bigl( N^{-}(b) \cup \{ b \} \bigr),\  (B\setminus A)_k \bigr),$

If $|S_k| = n_k \ge 2$, construct $n_k-1$ independent generators of trapezohedral type. 

Define $\omega_s^k$ as in Section \ref{sec:structureomega3},let
$\mathcal{B}_2 = \bigcup_{k=1}^{t}\{ \omega_s^k \mid 1 \le s \le n_k-1 \}.$

\item[(v)] Set
$\mathcal{B}_{a b} := \mathcal{B}_0 \cup \mathcal{B}_1 \cup \mathcal{B}_2.$

\end{enumerate}

\item[(II)] \textbf{Output: global basis construction.}

Define
$\mathcal{B} := \bigcup_{a,b \in V(G)} \mathcal{B}_{a b}.$ The set $\mathcal{B}$ forms a basis of $\Omega_3(G)$.

%\item[(III)] \textbf{Output.}

\end{enumerate}

\begin{proposition}
    The algorithm has time complexity $O(|V|^5)$.
\end{proposition}

\begin{proof}
    The algorithm proceeds by summing over all ordered pairs $(a,b) \in V \times V$. Thus it suffices to consider a fixed pair $a, b$ and show that a basis for $\Omega_3^{(a,b)}(G)$ can be computed in $O(|V|^3)$ time. 
    
    First, decomposing $V(G)$ into $A \cap B$, $A \setminus B$, and $B \setminus A$ can be done in $O(|V|)$ time. Finding a basis in $\mathcal{B}_1$, equivalent to enumerating all edges in $G[A \cap B]$ requires $O(|V|^2)$ time, as we need to check for each pair of vertices in $A \cap B$ whether they are connected by an edge.

    Second, to find a basis for $\mathcal{B}_0$, we construct the graph $H = \operatorname{Ind}(A \setminus B, B \setminus A)$, which can be done in $O(|V|^2)$ time. Computing its connected components via BFS or DFS takes $O(|V| + |E|) \leq O(|V|^2)$ time. We then compute spanning trees for all components of $H$ using Prim's algorithm, also in $O(|V|^2)$ time. During the execution of Prim's algorithm, we can simultaneously enumerate all shortest paths between every pair of vertices within each component. For each edge not belonging to any spanning tree, we obtain a corresponding generator in $\mathcal{B}_0$. Since there are at most $O(|V|^2)$ such edges and writing down each generator takes $O(|V|)$ time, this step requires $O(|V|^3)$ time in total.

    Third, to find a basis for $\mathcal{B}_2$, we enumerate all edges in each $\Delta_k^1$ and $\Delta_k^2$, which takes $O(|V|^2)$ time. For every pair of edges in $\Delta_k^1 \cup \Delta_k^2$, writing down the corresponding generator requires $O(|V|)$ time. Therefore, this step has a total time complexity of $O(|V|^3)$.

\end{proof}

%\section{Concluding Remarks}


\begin{thebibliography}{99}

\bibitem{chen2023path}
D. Chen, J. Liu, J. Wu, G.-W. Wei, F. Pan, and S.-T. Yau,
\emph{Path topology in molecular and materials sciences},
J. Phys. Chem. Lett. \textbf{14} (2023), 954--964.

\bibitem{di2024path}
S. Di, S. O. Ivanov, L. Mukoseev, and M. Zhang,
\emph{On the path homology of Cayley digraphs and covering digraphs},
J. Algebra \textbf{653} (2024), 156--199.

\bibitem{fu2024path}
X. Fu and S. O. Ivanov,
\emph{Path homology of digraphs without multisquares and its comparison with homology of spaces},
arXiv:2407.17001.

\bibitem{grigoryan2022advances}
A. Grigor'yan,
\emph{Advances in path homology theory of digraphs},
2022.

\bibitem{grigoryan_hodge}
A. Grigor'yan,
\emph{Path homology and Hodge Laplacian on digraphs},
\url{https://www.math.uni-bielefeld.de/~grigor/hodge2.pdf}

\bibitem{GrigorYan2019homology}
A. Grigor'yan, R. Jimenez, Yu. Muranov, and S.-T. Yau,
\emph{Homology of path complexes and hypergraphs},
Topol. Appl. \textbf{267} (2019), 106877.

\bibitem{GrigorYan2013homologies}
A. Grigor'yan, Y. Lin, Yu. Muranov, and S.-T. Yau,
\emph{Homologies of path complexes and digraphs},
arXiv:1207.2834v4 (2013).

\bibitem{GrigorYan2014homotopy}
A. Grigor'yan, Y. Lin, Yu. Muranov, and S.-T. Yau,
\emph{Homotopy theory for digraphs},
Pure Appl. Math. Q. \textbf{10} (2014), no. 4, 619--674.

\bibitem{GrigorYan2020path}
A. Grigor'yan, Y. Lin, Yu. Muranov, and S.-T. Yau,
\emph{Path complexes and their homologies},
J. Math. Sci. \textbf{248} (2020), no. 5, 564--599.

\bibitem{grigoryan2025eigenvalues}
A. Grigor'yan, Y. Lin, S.-T. Yau, and H. Zhang,
\emph{Eigenvalues of the Hodge Laplacian on digraphs},
Comm. Anal. Geom. \textbf{33} (2025), no. 4, 981--1023.

\begin{comment}
    \bibitem{GrigorYan2020torsions}
A. Grigor'yan, Y. Lin, and S.-T. Yau,
\emph{Analytic and Reidemeister torsions of digraphs and path complexes},
preprint (2020).

\bibitem{GrigorYanPacific}
A. Grigor'yan and Yu. Muranov,
\emph{On homology theories of cubical digraphs},
Pacific J. Math., to appear.
\end{comment}


\bibitem{GrigorYan2018multigraphs}
A. Grigor'yan, Yu. Muranov, V. Vershinin, and S.-T. Yau,
\emph{Path homology theory of multigraphs and quivers},
Forum Math. \textbf{30} (2018), no. 5, 1319--1337.

\bibitem{GrigorYan2014graphs}
A. Grigor'yan, Yu. Muranov, and S.-T. Yau,
\emph{Graphs associated with simplicial complexes},
Homology Homotopy Appl. \textbf{16} (2014), no. 1, 295--311.

\begin{comment}
    \bibitem{GrigorYan2015cohomology}
A. Grigor'yan, Yu. Muranov, and S.-T. Yau,
\emph{Cohomology of digraphs and (undirected) graphs},
Asian J. Math. \textbf{19} (2015), 887--932.

\bibitem{GrigorYan2016hochschild}
A. Grigor'yan, Yu. Muranov, and S.-T. Yau,
\emph{On a cohomology of digraphs and Hochschild cohomology},
J. Homotopy Relat. Struct. \textbf{11} (2016), no. 2, 209--230.
\end{comment}


\bibitem{GrigorYan2017kunneth}
A. Grigor'yan, Yu. Muranov, and S.-T. Yau,
\emph{Homologies of digraphs and K\"unneth formulas},
Comm. Anal. Geom. \textbf{25} (2017), no. 5, 969--1018.

\bibitem{ivanov2024simplicial}
S. O. Ivanov and F. Pavutnitskiy,
\emph{Simplicial approach to path homology of quivers, marked categories, groups and algebras},
J. London Math. Soc. \textbf{109} (2024), no. 1, e12812.

\bibitem{li2025singular}
J. Li, Y. Muranov, J. Wu, and S.-T. Yau,
\emph{On singular homology theories of digraphs and quivers},
J. Combin., to appear (2025).

\bibitem{lin2021discrete}
Y. Lin, C. Wang, and S.-T. Yau,
\emph{Discrete Morse theory on digraphs},
Pure Appl. Math. Q. \textbf{17} (2021), no. 5, 1711--1737.

\bibitem{liu2023neighborhood}
J. Liu, D. Chen, F. Pan, and J. Wu,
\emph{Neighborhood path complex for the quantitative analysis of the structure and stability of carbonanes},
J. Comput. Biophys. Chem. \textbf{22} (2023), no. 4, 503--511.

\bibitem{liu2023persistent}
R. Liu, X. Liu, and J. Wu,
\emph{Persistent path-spectral (PPS) based machine learning for protein–ligand binding affinity prediction},
J. Chem. Inf. Model. \textbf{63} (2023), 1066--1075.

\bibitem{tang2025minimal}
X. Tang and S.-T. Yau,
\emph{Minimal path and acyclic model in the path complex},
Comm. Anal. Geom. \textbf{33} (2025), no. 2, 275--342;
arXiv:2208.14063v2.

\bibitem{tang2024cellular}
X. Tang and S.-T. Yau,
\emph{The cellular homology of digraphs},
arXiv:2402.05682v2, accepted by Asian J. Math.

\end{thebibliography}
\end{document}